\documentclass[10pt,a4paper]{amsart}
\usepackage{geometry}

\usepackage{etoolbox}
\apptocmd{\sloppy}{\hbadness 10000\relax}{}{}

\newcommand{\comment}[1]{} 
\usepackage{microtype}
\usepackage{babel}
\usepackage[utf8]{inputenc}
\usepackage{enumerate}
\usepackage{amsmath,amsthm,amssymb,amsfonts}
\usepackage{tikz-cd}
\usetikzlibrary{arrows}
\input xy
\xyoption{all}
\newdir{ >}{{}*!/-5pt/@{>}}

\newcommand{\Z}{\mathbb{Z}}
\newcommand{\N}{\mathbb{N}}

\newcommand{\F}{\mathbb{F}}

\newcommand{\cC}{\mathcal C}
\newcommand{\cD}{\mathcal D}
\newcommand{\cE}{\mathcal E}

\newcommand{\cK}{\mathcal K}
\newcommand{\cL}{\mathcal L}

\newcommand{\cP}{\mathcal P}

\newcommand{\cR}{\mathcal R}
\newcommand{\cS}{\mathcal S}
\newcommand{\cT}{\mathcal T}
\newcommand{\cU}{\mathcal U}
\newcommand{\cV}{\mathcal V}

\newcommand{\topdf}{\texorpdfstring}

\newcommand{\BF}{\mathrm{BF}}

\newcommand{\xto}{\xrightarrow}
\newcommand{\iso}{\overset{\sim}{\longrightarrow}}
\DeclareMathOperator{\ad}{ad}
\DeclareMathOperator{\ab}{ab}
\DeclareMathOperator{\reg}{reg}

\DeclareMathOperator{\coker}{coker}

\DeclareMathOperator{\id}{id}
\DeclareMathOperator{\im}{im}
\DeclareMathOperator{\supp}{supp}
\DeclareMathOperator{\inc}{inc}

\DeclareMathOperator{\can}{can}
\DeclareMathOperator{\ev}{ev}

\DeclareMathOperator{\sink}{sink}
\DeclareMathOperator{\sour}{sour}

\DeclareMathOperator{\triv}{triv}
\DeclareMathOperator{\flip}{flip}
\DeclareMathOperator{\Idem}{Idem}
\def\gr{\operatorname{gr}}

\newcommand{\kkgr}{kk^{\gr}}
\newcommand{\cat}[1]{\mathsf{#1}}
\newcommand{\Alg}{\cat{Alg}_\ell}

\newcommand{\gBF}{\BF_{\gr}}
\newcommand{\elmat}{\varepsilon}

\newcommand{\grAlg}{\mathsf{Alg}^{\gr}}

\newcommand{\eps}{\epsilon}

\newcommand{\gc}{\widehat}
\newcommand{\isum}{\boxplus}

\usepackage{xcolor}

\definecolor{thmcol}{RGB}{186, 86, 4}
\definecolor{citecol}{RGB}{186, 86, 4}
\definecolor{linkcol}{RGB}{186, 86, 4}
\definecolor{urlcol}{RGB}{186, 86, 4}

\numberwithin{equation}{section}

\theoremstyle{plain}
\newtheorem*{thm*}{Theorem}
\newtheorem{thm}[equation]{Theorem}
\newtheorem{lem}[equation]{Lemma}
\newtheorem{coro}[equation]{Corollary}
\newtheorem{prop}[equation]{Proposition}
\newtheorem{conv}[equation]{Convention}

\theoremstyle{definition}
\newtheorem{defn}[equation]{Definition}
\newtheorem{conj}[equation]{Conjecture}
\newtheorem{ex}[equation]{Example}

\theoremstyle{remark}
\newtheorem{rmk}[equation]{Remark}

\newtheorem*{ack}{Acknowledgements}
\usepackage[colorlinks=true, pagebackref=true]{hyperref}
\hypersetup{
    colorlinks,
    citecolor=citecol,
    linkcolor=linkcol,
    urlcolor=urlcol
}
\usepackage{amsrefs}
\usepackage{url}

\title{Graded homotopy classification of Leavitt path algebras}
\author[Guido Arnone]{Guido Arnone}
\email{garnone@dm.uba.ar}

\address{Dep. Matem\'atica-IMAS\\ Facultad de Ciencias Exactas y Naturales\\
Universidad de Buenos Aires\\ Argentina}
\thanks{Supported by grant UBACyT 256BA from Universidad de Buenos Aires, PIP 2021-2023 GI 11220200100423CO from CONICET, PICT 2017-1395 from Agencia Nacional de Promoci\'on Cient\'\i fica y T\'ecnica, and by a PhD fellowship from CONICET}

\begin{document}

\begin{abstract}
We show that the graded Grothendieck group classifies  unital Leavitt path algebras of primitive graphs up to graded homotopy equivalence. 
To this end, we further develop classification 
techniques for Leavitt path algebras by means of (graded, bivariant) algebraic $K$-theory. 
\end{abstract}

\maketitle

\section{Introduction}

Let $\ell$ be a commutative unital ring. Given a 
graph $E$, we will consider its Leavitt path 
$\ell$-algebra $L(E)$ \cite{conmlpa}*{Definition 2.5}, which carries a natural grading
over $\Z$ (\cite{conmlpa}*{Proposition 4.7}).
Its graded Grothendieck group $K_0^{\gr}(L(E))$ is the group 
completion of the monoid of isomorphism classes of 
$\Z$-graded finitely generated projective modules. This group carries an action
from the infinite cyclic group $C_\infty = \langle \sigma\rangle$,
and is moreover a pointed preordered $\Z[\sigma]$-module; see 
Section \ref{subsec:ppom} for a precise defintion of these terms.
This paper is mainly concerned with the graded 
classification question for Leavitt path algebras:

\begin{conj}[\cite{hazrat}*{Conjecture 1}] \label{conj-hazrat} Assume that $\ell$ is a field. Let
$E$ and $F$ be two finite graphs.
If there exists a pointed, preordered $\Z[\sigma]$-module
isomorphism $K_0^{\gr}(L(E)) \iso K_0^{\gr}(L(F))$,
then the algebras $L(E)$ and $L(F)$ are isomorphic 
as graded algebras.
\end{conj}

In line with recent advances in the (ungraded) classification
question for purely infinite simple Leavitt path algebras 
\cites{kklpa1, kklpa2, kkhlpa}, we investigate the notion of 
graded classification up to polynomial homotopy. Before
stating our main result, we recall the relevant definitions and provide some motivation.
An elementary (graded, polynomial) homotopy between graded algebra homomorpshisms
$f,g \colon A \to B$ is
a graded homomorphism $h \colon A \to B[t]$
such that  $\ev_0 \circ \ h = f$ and $\ev_1 \circ \ h = g$;
here the indeterminate $t$ is set to be homogeneous of degree zero.
Two graded algebra maps are homotopic if they are connected
via finitely many elementary homotopies. 
Homotopy equivalences are then defined to be homomorphisms which
have an inverse up to this notion of homotopy.

A positive answer to Conjecture \ref{conj-hazrat} would
imply that two Leavitt path algebras 
are graded isomorphic if and only if they are
graded homotopy equivalent (see Remark \ref{rmk:htpy-khgr}).
One could thus first aim at deciding whether 
the graded Grothendieck group classifies 
Leavitt path algebras up to homotopy equivalence. 
In the ungraded, purely infinite simple case this program was carried out in 
\cite{kklpa2} and \cite{kkhlpa}.
The main theorem of this paper provides a similar result in the graded context:

\begin{thm}[Theorem \ref{thm:htpy-clasif}] \label{thm:main}  
Let $E$ and $F$ be two finite, primitive graphs. Assume that $\ell$ is a field.
If there exists a pointed, preordered $\Z[\sigma]$-module
isomorphism $K_0^{\gr}(L(E)) \iso K_0^{\gr}(L(F))$,
then the algebras $L(E)$ and $L(F)$ are graded homotopy
equivalent.
\end{thm}

We point out that despite the
resemblance of the theorem above with ungraded homotopy classification results, 
significant technical work is needed to obtain  
similar conclusions.
The hypothesis that graphs be \emph{primitive} (see Definition \ref{defn:primitive}) stems from adapting the techniques of \cite{kkhlpa}; put succinctly, we need a family of idempotents of $L(E)$ arising from edges to be full (Proposition \ref{prop:idem-primi}). 
Such graphs are in particular \emph{essential}, meaning that they have no sinks or sources. This allows us to interpret their associated Leavitt path algebras
as corner skew Laurent polynomial rings. The latter 
are characterized as the $\Z$-graded rings with a homogeneous left invertible element of degree $1$ (\cite{skew}*{Lemma 2.4}). These ideas are due to Ara and Pardo and go back to \cite{towards}.

The main tool used in this article is graded bivariant algebraic $K$-theory (\cite{kkg}). This is a functor $j \colon \grAlg \to \kkgr$ from the category 
of graded $\ell$-algebras to a triangulated category which universal in a specific sense; see Section \ref{sec:kkgr} for a brief recollection of its main properties. Here $\ell$
is viewed as a graded algebra with trivial grading. The functor 
$j$ is the identity on objects; hence we shall omit it from our notation. 
Writing $[L(E),L(F)]$ for the set of graded algebra homormophisms between two
Leavitt path algebras up to homotopy, our objective will be to understand 
the canonical map
\begin{equation}\label{canmap}
[L(E), L(F)] \to \kkgr(L(E), L(F)).
\end{equation}
In \cite{arcor}*{Corollary 11.11}, it is shown that homomorphisms between two Leavitt 
path algebras $L(E)$ and $L(F)$ in $\kkgr$ fit into an exact sequence involving 
their graded $K$-theory groups. Concretely, let $A_E$ be the adjacency matrix
of $E$ and $I$ the identity matrix on its set of vertices. Writing  
$\gBF(E) = \coker(I-\sigma A_E^t)$ for the \emph{Bowen-Franks module} of $E$ and 
$\gBF^\vee(E) = \coker(I^t-\sigma A_E)$ for its so-called dual,
there is a diagram of $\Z[\sigma]$-modules with exact top-row as follows:
\[
\begin{small}
\begin{tikzcd}[column sep = small]
    0 \arrow{r} & \gBF^\vee(E) \otimes_{\Z[\sigma]} K_1^{\gr}(L(F)) \arrow{r}{d} & \kkgr(L(E), L(F)) \arrow{r}{}
    & \hom_{\Z[\sigma]}(\gBF(E), K^{\gr}_0(L(F))) \arrow{r} & 0\\
    & & \text{$[L(E), L(F)]$} \arrow{u}{\overline j} \arrow[bend right = 30]{ru}{K^{\gr}_0} & &
\end{tikzcd}    
\end{small}
\]

The first part of this article is devoted 
to Poincaré duality for Leavitt path algebras,
which is used in \cite{kkhlpa}*{Lemma 13.1} to effectively 
compute the map $d$ in the exact sequence above.
Let us recall the relevant terminology. In this 
paper, a graph is a pair of source and range functions $s, r \colon E^1 \to E^0$
from a set of edges $E^1$ to a set of vertices $E^0$. The dual graph $E_t$
of $E$ has the same sets of vertices and edges with the functions $r$ and $s$ 
interchanged one for the other; informally, we revert the direction of all arrows.
The suspension in the triangulated structure of $\kkgr$ is represented
by tensoring by the trivially graded algebra $\Omega = t(1-t)\ell[t]$; 
we shall write $\Omega L(E_t)$ for $\Omega \otimes_\ell L(E_t)$.
\begin{thm}[Theorem \ref{thm:poinc}] If $E$ is a finite essential graph,
then the functor $- \otimes_\ell L(E)$ is 
left adjoint to $- \otimes \Omega L(E_t)$ as endofuntors 
of the graded bivariant $K$-theory category $\kkgr$. Thus,
for each pair of graded algebras $R$ and $S$ there are isomorphisms
\[
\kkgr(R \otimes_\ell L(E), S) \cong \kkgr(R,S\otimes_\ell \Omega L(E_t))
\]
which are natural in both $R$ and $S$.
\end{thm}

The proof of Poincaré duality given in \cite{kkhlpa} involves a specific 
homomorphism from $L(E)$ to the suspension algebra $\Sigma_X$ for a suitable set $X$. The latter algebra is a quotient of Karoubi's cone $\Gamma_X$ by the ideal of $X$-indexed matrices $M_X$.
In our context, we want this homomorphism to preserve
the gradings; in particular, we have to equip $\Sigma_X$
with a grading to begin with. Further, in ungraded algebraic bivariant $K$-theory tensoring by $\Sigma_X$ plays the role of the inverse for the suspension functor.
In Section \ref{sec:inf-sum} we generalize the notion of infinity sum-rings and, for a suitable notion of graded infinite set $X$, we define a graded analogue $\Gamma_X^{\gr}$ of Karoubi's cone and produce a quotient algebra $\Sigma_X^{\gr}$ which plays the role of the suspension in $\kkgr$. 

Another key ingredient of the proof of Poincaré duality is the 
relationship between a degree zero unit element $u_1 \in L(E) \otimes L(E_t)$ and the class in $\kkgr$ represented by the algebra homomorphism from $\cS := \ker(\ell[t, t^{-1}] \xto{\ev_1} \ell)$ to $L(E) \otimes L(E_t)$ mapping $t \mapsto u_1$. 
In this direction, we prove the following:
\begin{thm}[Theorem \ref{thm:rep-grunits}]
Let $A$ be a unital, strongly graded algebra, $p \in A_0$
an idempotent and $u$ a unit in $pA_0p$. Consider the map 
$\phi \colon \cS \to A$ given by $1 \mapsto p$, $t \mapsto u$.
There is a chain of isomorphisms
\[
\kkgr(\cS, A) \cong kk(\cS, A_0) \simeq KH_1(A_0),
\]
mapping $j(\phi)$ to $[1-p+u]$.
\end{thm}
Here $KH$ is Weibel's homotopy $K$-theory \cite{kh}. Finally, 
in Section \ref{sec:kkgr} we also record some
results that we use regarding 
left and right boundary maps of a triangle in $\kkgr$
and their compatibility with tensor products. 
We remark that these statements are not found in the literature, even in 
the ungraded case.

With a graded version of Poincaré duality in place, and the fact that $K_0^{\gr}$ is a full functor
when restricted to Leavitt path algebras (\cites{lift, vas}), we are 
able to study the image of \eqref{canmap}
(Lemmas \ref{lem:jf=dx+xi} and \ref{lem:jf+u=jfu}). This relies on a procedure to 
deform a unital graded algebra homomorphism $L(E) \to L(F)$ 
using a given element of $K_1^{\gr}(L(F))$, adapted from the ungraded setting in
Section \ref{sec:kkgr-to-alg}. Studying the extent to which
\eqref{canmap} is injective necessitates a study of $K_1^{\gr}(L(F))$. By a result of Hazrat \cite{gradedstr}*{Theorem 3.15} 
together with Dade's theorem \cite{dade}*{Theorem 2.8}, 
if $F$ is a graph with no sinks then one has canonical isomorphisms 
$K_*^{\gr}(L(F)) \cong K_\ast(L(F)_0)$. Further, 
it is known that $L(F)_0$ is an ultramatricial algebra \cite{amp}*{proof of Theorem 5.3}. Using these results, 
in \cite{towards}*{Lemma 3.6} Ara and Pardo give an explicit 
desctription of the shift action on 
$K_0(L(F)_0) \cong K_0^{\gr}(L(F))$. We 
extend the characterization above to $K_1(R_0) \cong K_1^{\gr}(R)$ for any strongly graded, corner skew Laurent polynomial ring $R$ such that $R_0$ is ultramatricial.
\begin{thm}[Corollary \ref{coro:d-shift-d=alpha}] \label{thm:main-shift=alpha}
Let $(R, t_+, t_-)$ be a strongly graded, corner skew Laurent polynomial ring. Assume that $\ell$ is a field and $R_0$ a unital ultramatricial algebra.
Writing $\alpha \colon R_0 \to R_0$ for the homomorphism given by 
$x \mapsto t_+ x t_-$, 
the following diagram is commutative:
\[
\begin{tikzcd}
    K_1^{\gr}(R) \arrow{r}{\sigma} & K_1^{\gr}(R)\\
    K_1(R_0) \arrow{u}{\sim} \arrow{r}{K_1(\alpha)} & \arrow{u}[right]{\sim} K_1(R_0)
\end{tikzcd}
\]
\end{thm} 
The proof involves the observation that $K_1(R_0)$ agrees with the abelianization of the unit group of $R_0$, recorded as Proposition \ref{prop:k1-ultra}. 
We also need an alternative description of the $K_1$ of an ultramatricial algebra,
akin to Cortiñas and Montero's characterization of Karoubi and Villamayor's $KV_1$ group
for purely infinite simple rings (\cite{kklpa2}*{Proposition 2.8}). 
\begin{thm}[Proposition \ref{prop:ultra-pi0}]
\label{thm:intro-kv1}
Assume that $\ell$ is a field. If $R$ is a unital ultramatricial algebra, 
then \[K_1(R) = R^\times/\{u(1) : u \in (R[t])^\times, u(0)=1\}.\]
\end{thm}
 
Theorem \ref{thm:intro-kv1} allows us to deduce injectivity of \eqref{canmap} up to 
a relaxed notion of homotopy. We say that two graded algebra maps 
$f,g \colon A \to B$
are \emph{ad-homotopic} if 
there exists a degree zero unit $u \in B_0^\times$ such that $f$ is homotopic to the conjugation of $g$ by $u$. A particular case of Theorem \ref{thm:jeq-ad-m2} is the following:
\begin{thm}
Let $E$ and $F$ be two primitive
graphs. Assume that $\ell$ is a field.
Two unital graded homomorphisms $f, g \colon L(E) \to L(F)$ satisfy $j(f) = j(g)$ 
if an only if they are ad-homotopic.
\end{thm}
With all of this in place, we
prove Theorem \ref{thm:main} in Section \ref{sec:clasi} as Theorem \ref{thm:htpy-clasif}.

\begin{ack}
I would like to thank Guillermo
Cortiñas for his careful reading of 
this manuscript as well as his
valuable suggestions.
Thanks also to Pere Ara and Enrique
Pardo for useful discussions
on the topics of Section \ref{subsec:ultra}.
\end{ack}

\section{Preliminaries}

For the rest of the article, we fix an abelian group $G$
and a commutative unital ring  $\ell$. 
The adjective graded will always mean $G$-graded. Recall that 
a \emph{graded algebra} is an $\ell$-algebra $R$ 
together with an abelian group descomposition $R = \bigoplus_{g \in G} R_g$ satisfying $R_g R_h \subset R_{gh}$ for each $g,h \in G$. 
The projection of $x \in R$ to $R_g$ will be denoted $x_g$.
If $x \in R_g$, then we say that $x$ is 
a \emph{homogeneous element of degree $g$} and write $|x| = g$. 
The base ring $\ell$ is viewed as a graded algebra by 
equipping it with the trivial grading, that is, we set $|\lambda| = 1_G$ for all $\lambda \in \ell$. A graded algebra homomorphism
$f\colon R \to S$ is an algebra homomorphism satisfying $f(R_g) \subset S_g$ for each $g \in G$. We shall denote the category of graded algebras by $\grAlg$. 

\subsection{Tensor products} Given two graded algebras $R$ and $S$, their tensor product has a canonical grading given 
by 
\[
(R\otimes_\ell S)_{d} = \bigoplus_{gh = d} R_g \otimes S_h.
\]

We equip the ring $\ell[t]$ of 
polynomials with the trivial grading, and set $R[t] := R \otimes_\ell \ell[t]$ with the tensor product grading
for any graded algebra $R$. In other words, we set $|t| = 1_G$.

We will often omit the tensor product symbol and use juxtaposition in its place, especially when one of the algebras carries the trivial grading. For example, we define 
the \emph{path} and \emph{loop} algebras respectively as
\[
P = \ker(\ell[t] \xto{\ev_1} \ell), \qquad \Omega = \ker(P \xto{\ev_0} \ell)
\]
and write 
\[
PR = P \otimes_\ell R, \qquad \Omega R = \Omega \otimes_\ell R
\]
for any graded algebra $R$.

\subsection{Graded sets, matricial stability and \topdf{$G$}{G}-stability} 

A \emph{graded set} will mean a pair $(X,d)$ where $X$
is a set and $d \colon X \to G$ is a function. When 
understood from context, we shall write $|\cdot|$ 
instead of $d$. An element $x \in X$
is said to have \emph{degree} $d(x) \in G$, and 
the degree $g \in G$ component of $X$ is $X_g := d^{-1}(g)$.
A morphism of graded sets $f \colon (X,d) \to (Y,d')$
is a function $f \colon X \to Y$ such that $d' \circ f = d$. If $X$ is a set, we write $|X|$
for the graded set given by the constant function with value $1_G$.

From a graded set $X$ one can form a graded algebra of 
\emph{$X$-indexed matrices}, denoted by $M_X$. As an algebra, it is the free $\ell$-module with basis $\{\elmat_{x,y} : x,y \in X\}$ with product $\elmat_{x,y} \cdot \elmat_{w,z} = \delta_{y,w} \elmat_{x,z}$. Its grading is induced by the assignment $|\elmat_{x,y}| = |x||y|^{-1}$. If $R$ is a graded algebra, we set $M_X R := M_X \otimes_\ell R$.

\begin{defn} We put $M_\infty$ for $M_{|\N|}$ and $M_n = M_{|\{1, \ldots, n\}|}$
for each $n \in \N$.
\end{defn}

Any morphism $f \colon X \to Y$ of graded sets with injective underlying function gives rise to a graded algebra map
$Mf \colon M_X \to M_Y$ mapping $\elmat_{x,y}$ to $\elmat_{f(x), f(y)}$.

\begin{defn}
Let $F \colon \grAlg \to \cat{C}$ be a functor. 
We say that $F$ is \emph{matricially stable}
if for every pair of sets $X, Y$ of cardinality
less or equal than $\beth = \max\{\aleph_0, |G|\}$ and any graded algebra $R$, 
the functor $F$ maps the inclusion $M_{|X|}R \to M_{|X \sqcup Y|}R$
to an isomorphism. If $F$ moreover maps 
inclusions $M_X R \to M_{X \sqcup Y}R$ to isomorphisms
for every pair of graded sets of cardinality less or equal than $\beth$, we say that it is \emph{$G$-stable}. 
\end{defn}

The following proposition is implied by \cite{tesigui}*{Proposición 3.3.8}.

\begin{prop}\label{prop:incx=incy} Let $X$ be a graded set
and $A$ a graded algebra. If
$x \in X$, then 
\[
\iota_x \colon a \mapsto \elmat_{x,x} \otimes a \in M_X A  
\]
is a graded homomorphism. Moreover, if $y \in X$ is such that $d(x) = d(y)$, then any $G$-stable functor $F$ satisfies $F(\iota_x) = F(\iota_y)$.
\qed
\end{prop}

Unless specified otherwise, we view
$G$ as a graded set via $\id_G \colon G \to G$. 
If $A$ is a graded algebra, we write $\iota_A$ for $\iota_{1_G} \colon A \to M_G A$.

\subsection{Extensions} An \emph{extension}
of graded algebras is an exact sequence
\[
K \xto{i} E \xto{p} Q
\]
such that
$p = \coker(i)$, $i = \ker(p)$, and 
$p$ admits a linear section $s \colon Q \to E$.

\begin{rmk} The section in the definition
of extension is not required to preserve the 
grading. This is justified by the fact that the existence of an $\ell$-linear section $s$ guarantees the existence of a section
\[
\hat s(m) = \sum_{d \in G} s(m_d)_d.
\]
which preserves the grading.
\end{rmk}

\begin{ex} \label{ex:pathext}
The \emph{loop extension} of a graded algebra
$R$ is 
\[
\Omega R \hookrightarrow PR \xto{\ev_1} R.
\]    
\end{ex}

\begin{ex} \label{ex:conext}
Let 
\begin{small}
\begin{equation}\label{eq:karcon}
\Gamma  = \{f \colon \N \times \N \to \ell : |\im f| < 
\infty \text{ and } (\exists N \geq 1) \text{ s.t. } |\supp(x,-)|, |\supp(-,x)| 
\leq N (\forall x \in X)\}
\end{equation}    
\end{small}
be Karoubi's cone ring, equipped with pointwise addition and the convolution product. We view $\Gamma$ as a graded
algebra via the trivial grading. Observe that it contains $M_\infty$ as an ideal; put $\Sigma = \Gamma/M_\infty$ for the suspension ring. The extension
\[
M_\infty R \to \Gamma R \to \Sigma R
\]
is the cone extension of $R$ (\cite{kk}*{Section 4.7}). We will deal with generalizations of 
this extension in Section \ref{sec:inf-sum}.
\end{ex}

\subsection{Homotopy invariance} 

In this paper we will only consider graded notions of algebraic
homotopy. This justifies dropping the adjective ``graded''
from all of the following definitions.

An \emph{elementary (polynomial) homotopy}
between graded algebra homomorphisms $f,g \colon A \to B$ is a
graded map $h \colon A \to B[t]$ such that $\ev_0 \circ h = f$, $\ev_1 \circ h = g$. 
We say that $f$ and $g$ are \emph{homotopic}
if there exists a sequence of elementary homotopies $h_1, \ldots h_n \colon A \to B[t]$
such that $\ev_0 \circ h_1 = f$, $\ev_1 \circ h_n = g$ and $\ev_1 \circ h_j = \ev_0 \circ h_{j+1}$ for all $j$. This is 
an equivalence relation which will be denoted $\sim$. 
Two graded (unital) algebras $R$ and $S$ are \emph{(unitally) 
homotopy equivalent} if there exist (unital) maps $f \colon R \to S$
and $g \colon S \to R$ such that $fg \sim 1_S$ and $gf \sim 1_R$.

\begin{defn} We say that a functor $F \colon \grAlg\to\cat{C}$
is homotopy invariant if it maps the inclusion $A \subset A[t]$
for each graded algebra $A$ to an isomorphism. Equivalently, 
a functor is homotopy invariant if $f \sim g$ implies $F(f) = F(g)$.
\end{defn}

We will also need two more notions of homotopy. 
The first one involves matricial stabilization. We say that 
two graded homomorphisms $f,g \colon A \to B$
are $M_2$-homotopic if the maps $\iota_1 \circ f, \iota_1 \circ g \colon A \to M_2 B$ are homotopic. As above, 
this also induces an equivalence relation denoted by~$\sim_{M_2}$. Likewise, we have a notion of \emph{$M_2$-homotopy equivalence}; we say that two algebras are \emph{$M_2$-homotopy equivalent} if there exists an $M_2$-homotopy equivalence between them.
For the second notion we first need a definition.

\begin{defn}\label{def:ad}
Let $C$ be a graded unital algebra and $A,B \subset C$ two subalgebras with inclusion maps $\inc_A \colon A\to C$ and $\inc_B \colon B\to C$.
Given $u,v \in C$ two homogeneous elements such that $|u||v| = 1$, $avua' = aa'$ 
for each $a,a' \in A$ and $uAv\subset B$, we define the graded homomorphism
\[
\ad(u,v) \colon A \to B, \qquad a \mapsto uav.
\]
If $u$ is a unit, we abbreviate $\ad(u) := \ad(u,u^{-1})$.
\end{defn}

We say that two unital algebra maps $f,g \colon R \to S$
are \emph{$\ad$-homotopic}, denoted $f \sim_{\ad} g$ if there exists a unit $u \in S_{1_G}$ such that $\ad(u) \circ f \sim g$. When this happens, we will
write $f \sim_u g$. The proof of the 
following lemma is implied by Proposition \ref{prop:ad} below.

\begin{lem} \label{lem:ad->m2} If two unital graded
algebra maps are $\ad$-homotopic, they are $M_2$-homotopic.
\qed
\end{lem}

\begin{prop}\label{prop:ad} Let $C$ be a graded unital algebra and $A,B \subset C$ two subalgebras with inclusion maps $\inc_A \colon A\to C$ and $\inc_B \colon B\to C$.
Given $u,v \in C$ in the situation of Definition \ref{def:ad}, we have that:
\begin{itemize}
    \item[i)] any $G$-stable functor $F$ satisfies $F(\inc_B \ad(u,v)) = F(\inc_A)$;
    \item[ii)] if $B = A$, $uA, Av \subset A$, and $|u| = 1_G$, then $\ad(u,v) \sim_{M_2} \id_A$. 
\end{itemize}
In particular, by ii), any homotopy invariant, matricially stable functor $F$ satisfies $F(\ad(u,v)) = \id_{F(A)}$.
\end{prop}
\begin{proof}
For the proof of i), we adapt \cite{friendly}*{Proposition 2.2.6}.
Put $d := |u|$ and note that $|v| = d^{-1}$. The assignment 
$1 \mapsto 1$, $2 \mapsto d$ yields a grading on $M_2$; 
in the rest of the proof we will consider the latter algebra equipped with
this particular grading.

Put $U = \elmat_{1,1} u + \elmat_{2,2} 1_C$ and $V = \elmat_{1,1}v + \elmat_{2,2}1_C$.
A straightforward verification shows that we have a well-defined graded
algebra homomorphism
\[
\phi \colon M_2 A \to M_2 C, \qquad x \mapsto UxV.
\]
For each $k \in \{1,2\}$, 
write $\iota_k^R \colon R\to M_2 R$ for the corner inclusions of each algebra $R$. Observe that
\[
\phi \iota_1^A = \iota_1^C \inc_B \ad(u,v), \qquad \phi \iota_2^A = M_2(\inc_A)\iota_2^A.
\]
Applying $F$ we obtain that $F(\phi)F(\iota_2^A) = F(M_2(\inc_A))F(\iota_2^A)$. Since $F(\iota_2^A)$ is an isomorphism by $G$-stability, it follows that $F(\phi) = F(M_2(\inc_A))$. 
Hence 
\[
F(\iota_1^C)F(\inc_B\ad(u,v)) = F(\phi\iota_1^A) = F(M_2(\inc_A)\iota_1^A) = F(\iota_1^C)F(\inc_A).
\]
Once again by $G$-stability, we know that $F(\iota_1^C)$ is 
an isomorphism and thus $F(\inc_B\ad(u,v)) = F(\inc_A)$. 
This concludes the proof of i).

Now suppose that $A = B$, $uA, Av \subset A$ and $|u|=1_G$. Write $\iota_k = \iota_k^A$. The hypotheses imply in particular
that $\phi$ can be correstricted 
to a homomorphism $\psi \colon M_2 A \to M_2 A$
satisfying $\psi \iota_1 = \iota_1\ad(u,v)$ and $\psi \iota_2 = \iota_2$.
They also say that the grading on $M_2 A$ is the standard one and that
the homotopy given \cite{kklpa1}*{Lemma 2.1} is a graded homotopy between $\iota_1$ and $\iota_2$. This implies that
$\iota_1 \ad(u,v) = \psi \iota_1 \sim \psi\iota_2 = \iota_2 \sim \iota_1$,
proving ii).
\end{proof}

\subsection{Graded \topdf{$K$}{K}-theory}

Given a unital graded algebra $R$, 
its \emph{graded $K$-theory} $K_\ast^{\gr}(R)$ is the $K$-theory of the exact category of graded, finitely generated projective $R$-modules. 
In \cite{arcor}*{Section 3.3}, a homotopy invariant version $KH^{\gr}$ of graded $K$-theory is introduced. It comes equipped with a canonical comparison map $K^{\gr}_\ast(R) \to KH^{\gr}_\ast(R)$ for any graded algebra $R$.
Given a graded left $R$-module $M$, its \emph{shift} by $g \in G$ is the module $M[g] := M$
with the grading $M[g]_h = M_{gh}$. Note that the 
shift of modules induces an action of $\Z[G]$ on graded (homotopy) 
$K$-theory. In particular, 
the graded (homotopy) $K$-theory of a $\Z$-graded
algebra is a $\Z[\sigma]$-module.

\subsection{Strongly graded rings}

A graded unital ring $R$ is \emph{strongly graded}
if $R_g R_h = R_{gh}$ for each $g,h \in G$.
A theorem
of Dade \cite{dade}*{Theorem 2.8} says that $R$ is strongly 
graded if and only if the functor
\[
R \otimes_{R_{1_G}} - \colon \cat{Proj}_{\cat{fg}}(R_{1_G}) \to \cat{Gr-Proj}_{\cat{fg}}(R)
\]
is an equivalence of categories with inverse
\[
(-)_{1_G} \colon \cat{Gr-Proj}_{\cat{fg}}(R)
\to \cat{Proj}_{\cat{fg}}(R_{1_G}).
\]
In particular one has
canonical isomorphisms
\begin{equation}\label{map:dade}
K_\ast(R_{1_G}) \to K^{\gr}_\ast(R), 
\qquad 
KH_\ast(R_{1_G}) \to KH^{\gr}_\ast(R).
\end{equation}
for every strongly graded algebra $R$.

\begin{ex} \label{ex:tensor-strong}
By \cite{hazrat}*{Example 1.1.16}, if $A$ is a strongly graded unital algebra then so is $B \otimes_\ell A$ for any graded unital algebra $B$. 
\end{ex}

\subsection{Graphs, their Leavitt path algebras and Bowen-Franks modules}\label{subsec:lpa}

A (finite, directed) \emph{graph} is a tuple $E = (E^0,E^1,r,s)$
consisting of two finite sets $E^0$ of \emph{vertices}
and $E^1$ of \emph{edges} together with \emph{range}
and \emph{source} functions $r,s \colon E^1 \to E^0$.
We say that $v \in E^0$ is a \emph{sink} if $s^{-1}(v) = \emptyset$,
\emph{regular} if it is not a sink, and a \emph{source} if $r^{-1}(v) = \emptyset$. The sets of sinks, regular vertices and sources are denoted by $\sink(E)$, $\reg(E)$ and $\sour(E)$
respectively. 
A graph $E$ 
is \emph{regular} if $E^0 = \reg(E)$ and \emph{essential}
if it is regular and $\sour(E) = \emptyset$.

To a graph $E$, we will associate its \emph{Leavitt path $\ell$-algebra} $L(E)$. The latter is a quotient of the free algebra on the set $\{v,e,e^\ast : v \in E^0, e \in E^1\}$ by the relations:
\begin{align}
&vw = \delta_{v,w}v,  &(v,w \in E^0), \tag{V}\label{V}\\
&s(e)e = er(e) = e, &(e \in E^1),\tag{E1}\label{E1}\\
&r(e)e^\ast = e^\ast s(e) = e^\ast, &(e \in E^1),\tag{E2}\label{E2}\\
&f^\ast e = \delta_{f,e}r(e),   &(f,e \in E^1),\tag{CK1}\label{ck1}\\
&v = \sum_{e \in s^{-1}(v)} ee^\ast,  &(v \in \reg(E)).\tag{CK2}\label{ck2}
\end{align}
The \emph{standard grading} of $L(E)$ over $\Z$ is given by extension of
the rule $|v| = 0$, $|e| = 1$, $|e^\ast| = -1$ for each $v \in E^0$, $e \in E^1$.
The \emph{Cohn algebra} of $E$ is the  one obtained similarly dividing by all of the relations above except \eqref{ck2}. One has a canonical surjection $C(E) \to L(E)$. Writing $q_v = v-\sum_{e \in s^{-1}(v)} ee^\ast$ for each $v \in \reg(E)$ and $\cK(E) := \langle q_v : v \in \reg(E)\rangle$, we have an
exact sequence
\begin{equation}\label{ext:cohn}\tag{$\cC_E$}
    \cK(E) \to C(E) \to L(E).
\end{equation}
We shall refer to the exact sequence above as the \emph{Cohn extension} of $L(E)$. By \cite{lpabook}*{Proposition 1.5.11}, it
is always $\ell$-linearly split.

The (reduced) adjacency matrix $A_E \in \N_0 \in \Z^{\reg(E) \times E^0}$ of a graph $E$ is the one given by 
\[
(A_E)_{v,w} = \#\{e \in E^1 : s(e) = v, r(e)=w\}.
\]
Writing $I$ for the $\reg(E)\times E^0$-indexed matrix given by $I_{v,w} = \delta_{v,w}$, the \emph{Bowen-Franks module} of $E$ is defined as
\[
\gBF(E) := \coker(I-\sigma A_E^t) = \frac{\Z[\sigma]^{E^0}}{\langle v - \sigma \sum_{e \in s^{-1}(v)} r(e) : v \in \reg(E)\rangle}.
\]

By \cite{arcor}*{Theorem 5.3}, the (homotopy) graded
$K$-theory of $L(E)$ can be computed in terms of this group as
\[
K_\ast^{\gr}(L(E)) = \gBF(E) \otimes_\Z K_\ast(\ell), \qquad KH_\ast^{\gr}(L(E)) = \gBF(E) \otimes_\Z KH_\ast(\ell).
\]

\begin{rmk}\label{rmk:htpy-khgr} 
By \cite{arcor}*{Theorem 3.9, Equation 3.9 and Theorem 5.3}, if $E$
is a finite graph, then the comparison map 
$K^{\gr}_\ast(L(E))\to KH^{\gr}_\ast(L(E))$ can be 
identified with tensoring the canonical comparison map $K_\ast(\ell) \to KH_\ast(\ell)$ by $\gBF(E)$. In particular, if $\ell$ is a field, a PID, or 
more generally a regular noetherian ring, then $K^{\gr}_\ast(L(E)) = KH_\ast^{\gr}(L(E))$.

In the case in which $\ell$ is a field, as a corollary of the above we obtain that any homotopy equiavlence $L(E) \to L(F)$ induces
an isomorphism at the level of $K_0^{\gr}$. Were 
Conjecture~\ref{conj-hazrat} to be true, this would entail
that the algebras $L(E)$ and $L(F)$ are graded isomorphic. In
other words, Conjecture \ref{conj-hazrat} would imply
that two Leavitt path algebras are graded isomorphic if and only
if they are graded homotopy equivalent.
\end{rmk}

\begin{rmk} \label{rmk:lpa-str}
If $E$ is a finite graph, then $L(E)$ endowed with its standard $\Z$-grading is
strongly graded if and only if the underlying
graph has no sinks (\cite{gradedstr}*{Theorem 3.15}).   
\end{rmk}

\subsection{Pointed preordered modules}\label{subsec:ppom}

A \emph{pointed preordered module} is a tuple $(M, M_+, u)$
where $M$ is a $\Z[\sigma]$-module together 
with a distinguished 
submonoid $M_+$ such that $\N_0[\sigma] M_+ \subset M_+$
and a distinguished element $u \in M_+$ that is an order unit; this means
that for every $m \in M$ there exists $x \in \N_0[\sigma]$
satisfying $x \cdot u - m \in M_+$. A pointed preordered module
map $f \colon (M,M_+,u) \to (N, N_+,v)$ is a $\Z[\sigma]$-linear
map $f \colon M \to N$ such that $f(M_+) \subset N_+$ and $f(u) = v$.

\begin{ex}
If $R$ is a graded ring, then its graded Grothendieck group 
together with the submonoid $K_0^{\gr}(L(E))_+$ of classes of projective modules and the class $[R]$ form a pointed preordered module.     
\end{ex}

\begin{ex}
The Bowen-Franks module of a graph $E$ can be made into
a pointed preordered module by setting $\gBF(E)_+ = \langle \sum_{v \in E^0} x_v \cdot v : x_v \in \N_0[\sigma] \rangle$ and $1_E := \sum_{v \in E^0} [v]$.    
\end{ex}

\begin{rmk}
There is a canonical pointed preordered module 
map
\[
\can \colon \gBF(E) \to K_0^{\gr}(L(E)), \qquad [v] \mapsto [vL(E)]
\]
which, by \cite{arcor}*{Corollary 5.4}, is an isomorphism whenever $K_0(\ell) \cong \Z$.    
\end{rmk}

\section{The category \topdf{$\kkgr$}{kkgr} and its triangulated structure}\label{sec:kkgr}

Recall that an \emph{excisive homology theory} for graded algebras is a functor 
$H \colon \grAlg \to \cT$ with codomain a triangulated category $\cT$, such that 
for each extension
\begin{align*} \label{exte} \tag{$\mathcal E$}
K \xto{i} E \xto{p} Q
\end{align*}
there exists a triangle
\[
H(Q)[+1] \xto{\partial^H_{\cE}} H(K) \xto{H(i)} H(E) \xto{H(p)} H(Q).
\]
The maps $\partial^H_{\cE}$ are called the \emph{(left) boundary map}
of $H$ associated to an extension $\cE$, and they are required
to satisfy some compatibility conditions (\cite{kk}*{6.6}).

A morphism of excisive homology theories $(F,\phi)$ from $H$ to $H' \colon \grAlg \to \cT'$ consists of a triangulated functor  $F \colon \cT \to \cT'$ such that $FH = H'$
and a natural transformation $\phi \colon F(H(-)[+1]) \to H'(-)[+1]$ such that
the following diagram commutes for all extensions $\cE$:
\begin{equation} \label{mor:eht}
\begin{tikzcd}
    F(H(Q)[+1]) \arrow{dd}{\phi_Q} \arrow{dr}{F(\partial^H_{\cE})}& \\
     & H'(K) \\
     H'(Q)[+1] \arrow{ur}[below]{\qquad \quad \partial_{\cE}^{H'}} &   
\end{tikzcd}
\end{equation}
Graded bivariant algebraic $K$-theory
\[
j \colon \grAlg \to \kkgr
\]
is the initial $G$-stable, matricially stable, homotopy invariant excisive homology theory (\cite{kkg}*{Theorem 4.2.1}). We refer
the reader to \cite{kkg} (see also \cite{arcor}*{Sections 7 and 8}) for the construction of the category $\kkgr$ and its main properties.

In this section we shall study properties of 
boundary maps $\partial_{\cE} := \partial^j_{\cE}$ in $\kkgr$. 
The construction of $\kkgr$ is built off
that of the initial homology theory among the matricially stable, homotopy invariant theories which are not necessarily $G$-stable. This functor
will be denoted $j' \colon \grAlg \to kk_{\grAlg}$.
We first give an account on 
boundary maps in the latter theory.

\subsection{Boundary maps in \topdf{$kk_{\grAlg}$}{kk-grAlg}}

The category $kk_{\grAlg}$ is constructed, \textit{mutatis mutandis}, 
in the same way algebraic bivariant $K$-theory is constructed in \cite{kk}; we refer
the reader to \cite{kkg}*{Section 2} for a detailed explanation on the construction
of these categories. Its objects are, as those of $\kkgr$, all graded algebras.

Recall that given an extension \eqref{exte}, its \emph{classifying map} $c_\cE$ is computed explicitly by considering the tensor algebra
map $TQ \to E$ induced by a section $s \colon Q \to E$ of $p$, 
and then restricting it to the kernel $JQ := \ker(TQ \to Q)$
of the counit map $T \Rightarrow \id$. Up to homotopy $c_{\cE}$ is 
independent of the chosen section. 

 The boundary map $\partial'_{\cE} := \partial^{j'}_{\cE}$ of
an extension \eqref{exte} is given by a zig-zag of classifying maps, that of $\cE$ and the one associated to the \emph{loop extension}
\[\tag{$\cL_Q$}
\Omega Q \to PQ \xto{\ev_1} Q.
\]
Namely,
\[
\partial'_{\cE} := c_{\cE} \circ c_{\cL_{Q}}^{-1} = \Omega Q \xleftarrow{c_{\cL_{Q}}} J Q \xto{c_{\cE}} K.
\] 

\begin{rmk} \label{rmk:uht-loop}
Note that, by construction, the boundary map of $\cL_Q$ 
is the identity map of $\Omega Q$. As recalled 
above, if $H \colon \grAlg \to \cT$
is a matricially stable, homotopy invariant 
excisive homology then there is a unique homomorphism 
$(X,\phi) \colon j' \to H$.
Although we will not 
delve into the construction of $X$, 
we nonetheless note that since $\partial'_{\cL_Q} = \id_{\Omega Q}$
it follows that $\phi_Q = (\partial_{\cL_Q}^H)^{-1}$ for all $Q$. 
\end{rmk}

\subsection{Boundary maps in \topdf{$\kkgr$}{kkgr}}

The category $\kkgr$ is built in terms of $kk_{\grAlg}$.
Given $A, B \in \grAlg$, 
by definition
\[
\kkgr(A,B) = kk_{\grAlg}(M_G A, M_G B).
\]
If $f \colon A \to B$ is a graded algebra homomorphism, 
then $j(f) = j'(M_G f)$. 
The boundary map of an extension \eqref{exte} is, thus, 
an element $\partial_{\cE} \in \kkgr(\Omega Q, K) = 
kk_{\grAlg}(M_G \Omega Q, M_G K)$. To construct it, we consider the extension
\begin{align*} \label{MG-exte} \tag{$M_G \cE$}
M_G K \xto{M_G i} M_G E \xto{M_G p} M_G Q.
\end{align*}
Its boundary map in $kk_{\grAlg}$ 
is an arrow $\partial'_{M_G \cE} \colon \Omega M_G Q \to M_G K$. To define $\partial_{\cE}$,  
we precompose the latter by the by the flip map
$\tau_Q \colon M_G\Omega Q \to \Omega M_G Q$:
\begin{equation}\label{def:partial}
\partial_{\cE} := \partial'_{M_G \cE} \circ \tau_Q = c_{M_G \cE} \circ c_{\cL_{M_G Q}}^{-1} \circ \tau_Q.
\end{equation}

Our first computation will concern the loop extension. 

\begin{lem} \label{lem:partial-loop}
If $A$ is a graded algebra,
then $\partial_{\cL_A} = \id_{\Omega A}$.
\end{lem}
\begin{proof} 
Put $\cE = M_G \cL_A$ and $\cD = \cL_{M_G A}$. Following the definition of \eqref{def:partial}, 
we obtain the following equality in $kk_{\grAlg}$:
\[
\partial_{\cL_A} = \partial'_{M_G \cL_A} \circ \tau_{A}
= c_{\cE} \circ c_{\cD}^{-1} \circ \tau_A.
\]
Consider the flip 
map $\tau'_A \colon M_G P A \to P M_G A$. It fits in 
a morphism of extensions from $\cE$ to $\cD$:
\[
\begin{tikzcd}
M_G \Omega A \arrow{d}{\tau_A} \arrow{r}
& \arrow{d}{\tau'_A} M_G P A \arrow{r} & 
M_G A \arrow[equals]{d}  \\
\Omega M_G A \arrow{r} & P M_G A \arrow{r} & M_G A 
\end{tikzcd}
\]
By the graded version of \cite{kk}*{Proposition 4.4.2} (see \cite{kkg}*{p. 205-206}),
in $kk_{\grAlg}$ we have the equality $\tau_A \circ c_{\cE} = c_{\cD}$.
Hence $c_{\cE}$ is an isomorphism and $c_{\cD}^{-1} = c_{\cE}^{-1}\tau_A^{-1}$; this 
concludes the proof. 
\end{proof}

\begin{rmk} \label{rmk:kkg-uht-loop}
Let $H \colon \grAlg \to \cT$ be a $G$-stable, matricially stable, homotopy
invariant excisive homology theory. We recall the construction of
the unique map $j \to H$ from \cite{kkg}*{Section 4.2}. 

Given 
the unique map $(X',\phi') \colon j' \to H$, one defines $X(A) = X'(A)$
on objects $A \in \kkgr$ and $X(\alpha) = X'(\iota_B)^{-1}X'(\alpha)X'(\iota_A)$
on morphisms $\alpha \colon A\to B$.
The natural transformation $\iota \colon \id \to M_G(-)$ induces, 
for any extension \eqref{exte},
a map of extensions $\cE \to M_G \cE$. In particular, 
it follows that $\iota_K \circ \partial'_{\cE} =  
\partial'_{M_G \cE} \circ  \Omega\iota_Q = \partial'_{M_G \cE} \circ \tau_Q \circ  \iota_{\Omega Q}$ in $kk_{\grAlg}$ and hence 
\[
X(\partial_{\cE}) = X'(\iota_K)^{-1} 
X'(\partial'_{M_G \cE} \circ \tau_{Q})X'(\iota_{\Omega Q}) = X'(\partial'_{\cE}).
\]
This automatically implies that setting $\phi_Q = \phi'_Q = (\partial^H_{\cL_Q})^{-1}$
makes $(X,\phi)$ into a morphism of excisive homology theories. 
Note also that, by Lemma \ref{lem:partial-loop},
this is the only possible choice for $\phi$; c.f. Remark \ref{rmk:uht-loop}.
\end{rmk}

\subsection{Ungraded extensions}

We will write $j_{kk} \colon \Alg \to kk$ for
ungraded algebraic bivariant $K$-theory \cite{kk}. 
There is a canonical map $\mathrm{triv} \colon kk \to \kkgr$ induced by the trivial grading inclusion $\triv \colon \Alg \hookrightarrow \grAlg$. 

In particular, one may view any extension of ungraded algebras 
as one of trivially graded ones. 
Since $j \circ \mathrm{triv} \colon \Alg \to \kkgr$ is a an 
excisive homology theory (for ungraded algebras, i.e. for $G = \{1\}$), 
there is a unique map $(X,\phi) \colon 
j_{kk} \to j\circ\triv$ and 
$\phi_Q = (\partial_{\triv(\cL_Q)})^{-1} = \partial_{\cL_{\triv(Q)}}^{-1} = \id_{\Omega Q}$ for each algebra $Q$. Thus, we have the following.

\begin{thm} \label{thm:triv-boundary}
Let $\triv \colon kk \to \kkgr$ be the canonical 
functor induced by the trivial grading functor $\triv \colon \Alg \hookrightarrow \grAlg$.
If $\cE$ is an extension of ungraded algebras and $\partial^{kk}_{\cE}$
its boundary map in $kk$, then $\triv(\partial^{kk}_{\cE}) = \partial_{\cE}$.
\qed
\end{thm}

We conclude this subsection with a characterization 
of the boundary map of the cone extension 
\[\label{ext:cone}
M_\infty \to \Gamma \to \Sigma \tag{$\mathfrak K$}
\]
in both 
the graded and ungraded case. First, we need a definition.

\begin{defn} Let $s_0 = \sum_{i \in \N} \elmat_{i+1,i} \in \Gamma$ be the \emph{right shift} and $s := [s_0]$ its class as an element of $\Sigma$. Since $s_0^\ast s_0 = 1$
and $s_0 s_0^\ast = 1-\elmat_{1,1}$, it follows that $s$ is a unit. Write $L = s^{-1}$ and
\[
\xi_L \colon \ell \to \Omega \Sigma
\]
for the morphism in $kk(\ell, \Omega \Sigma)$ corresponding to $[L] \in KH_1(\Sigma)$.
\end{defn}

\begin{lem} \label{lem:partial-s=e11}
The following diagram commutes in $\kkgr$:
\[
\begin{tikzcd}
\ell \arrow{r}{\triv(\xi_L)} \arrow{dr}[below]{\inc_1 \quad }& \Omega \Sigma \arrow{d}{\partial_{\mathfrak K}}  \\ & M_\infty
\end{tikzcd}
\]
In particular $\triv(\xi_L)$ is an isomorphism.
\end{lem}
\begin{proof} Since by Theorem \ref{thm:triv-boundary} the functor $\triv \colon kk \to \kkgr$
is compatible with boundary maps, we may assume that $G$ is the trivial group. The result now follows from \cite{kkhlpa}*{proof of Lemma 11.1}, 
which in particular says that the boundary $\partial \colon KH_1(\Sigma) \to KH_0(M_\infty)$ maps $[s]$ to $[1-s_0^\ast s_0] - [1-s_0s_0^\ast] = -[\elmat_{1,1}]$ and thus $\partial([L]) = [\elmat_{1,1}]$. 
\end{proof}

\subsection{Compatibility with tensor products.}

Next we use Lemma \ref{lem:partial-loop} to prove the 
compatibility of left boundary maps with tensor products.

\begin{thm} \label{thm:tensor-compatibility}
If
\[
K \to E\to Q \tag{$\cE$}
\]
is an extension and $A$ a graded algebra, then the boundary 
map of the extension 
\[
K \otimes A \to E \otimes A \to Q \otimes A \tag{$\cE \otimes A$}
\]
equals that of $\cE$ tensored by $A$. That is,
\[
\partial_{\cE \otimes A} = \partial_\cE \otimes A.
\]
\end{thm}
\begin{proof}
Recall that the functor $- \otimes A \colon \kkgr \to \kkgr$
is defined using the universal property of $j$
as the unique morphism of homology theories
from $j$ to $H := j(- \otimes A)$. 
As noted in Remark~\ref{rmk:kkg-uht-loop}, this entails in particular 
that $\partial_{\cE} \otimes A = \partial_{\cE}^H \circ (\partial_{\cL_Q}^{H})^{-1} = 
\partial_{\cE \otimes A} \circ \partial_{\cL_Q \otimes A}$.
Since 
$\mathcal{L}_Q \otimes A = \mathcal L_{Q \otimes A}$, it follows from Lemma 
\ref{lem:partial-loop} that $\mathcal{L}_Q \otimes A$ is the identity map; this concludes the proof.
\end{proof}

\subsection{Adjoint equivalences between \topdf{$\Omega$}{loops} and \topdf{$\Sigma$}{suspensions}} 

As a consequence of Lemma \ref{lem:partial-s=e11}, the natural transformation
\begin{equation}\label{def:unit}
\lambda \colon \Omega \Sigma \Rightarrow \id, \qquad \lambda_A := \triv(\xi_L)^{-1} \otimes A
\end{equation}
is an isomorphism. Put $\flip \colon \Omega \Sigma \cong \Sigma \Omega$
for the permutation of tensor factors. Since $\flip \otimes -$ is a natural isomorphism, so is $\gamma := (\flip \otimes - ) \circ \lambda^{-1} \colon \id \Rightarrow \Sigma \Omega$. We have thus explictly constructed pseudoinverses exhibiting the fact that
tensoring by $\Omega$ and $\Sigma$ yield inverse equivalences of categories.
In particlular, this allows us to see $\Omega$ as a left adjoint of $\Sigma$
by viewing $\lambda$ as the counit of an adjunction:

\begin{thm} \label{thm:loop-adj-ol}
The natural transformation $\lambda$ of \eqref{def:unit}
is the counit of an adjunction whose unit is $\Theta := \gamma^{-1} \Sigma \Omega \circ \Sigma \lambda^{-1} \Omega \circ \gamma$. Explicitly,
\begin{equation}\label{def:counit}
\Theta_B := \Theta_0 \otimes B, \qquad \Theta_0 = (\triv(\xi_L)^{-1} \circ \flip \otimes \Sigma \otimes \Omega) \circ (\Sigma \otimes \triv(\xi_L) \otimes \Omega) \circ (\flip \circ \triv(\xi_L)).
\end{equation}
\end{thm}
\begin{proof} This follows from the
characterization of an adjunction in terms of triangle identities (\cite{context}*{Remark 4.2.7}); the reader can view the dual construction of a counit
in \cite{context}*{proof of Proposition 4.4.5}.
\end{proof}

Similarly, we can use the natural equivalence 
$\mathfrak u := \lambda^{-1} = \triv(\xi_L) \otimes -$ with inverse $\gamma^{-1}$
to construct an adjunction in which $\Omega$ is right adjoint to $\Sigma$:
\begin{thm} \label{thm:loop-adj-or}
The natural transformation $\mathfrak u$, inverse to \eqref{def:unit},
is the unit of an adjunction whose counit is $\mathfrak c := \gamma^{-1} \circ \Sigma \lambda \Omega \circ \Sigma \Omega \gamma$. Explicitly,
\begin{equation}\label{def:unit'}
\mathfrak c_B = \mathfrak c_0 \otimes B, \qquad \mathfrak c_0 := (\xi_L^{-1}\circ \flip^{-1}) \circ (\Sigma \otimes \xi_L^{-1} \otimes \Omega) \circ (\Sigma \otimes \Omega \otimes \flip\circ\xi_L).
\end{equation}
\end{thm}

From Theorems \ref{thm:loop-adj-ol} and \ref{thm:loop-adj-or} we obtain, for each pair of graded algebras $A$ and $B$, natural abelian group isomorphisms
\begin{equation}\label{def:cR}
    \cR_{A,B} \colon \kkgr(\Omega A, B) \iso \kkgr(A, \Sigma B), \quad \zeta \mapsto \Sigma \zeta \circ \Theta_A,
\end{equation}
\begin{equation}\label{def:cL}
    \cL_{A,B} \colon \kkgr(A, \Sigma B) \iso \kkgr(\Omega A, B), \quad 
    \zeta \mapsto \lambda_B \circ \Omega \zeta
\end{equation}
and 
\begin{equation}\label{def:cU}
    \cU_{A,B} \colon \kkgr(\Sigma A, B) \iso \kkgr(A, \Omega B), \quad \zeta \mapsto \Omega \zeta \circ \mathfrak u_A,
\end{equation}
\begin{equation}\label{def:cV}
    \cV_{A,B} \colon \kkgr(A, \Omega B) \iso \kkgr(\Sigma A, B), \quad 
    \zeta \mapsto \mathfrak c_B \circ \Sigma \zeta.
\end{equation}

\subsection{Right boundaries}
Given an extension \eqref{exte}, its \emph{right boundary map} is defined as
\[
\delta_\cE := -\cR_{Q,K}(\partial_\cE) = -\Sigma \partial_\cE \circ (\Theta_0 \otimes Q).
\]

\begin{rmk} \label{rmk:right-boundary-tensor}
By Theorem \ref{thm:tensor-compatibility}, right boundary maps are compatible with tensoring in the sense that $\delta_{\cE \otimes A}  = \delta_\cE \otimes A$.    
\end{rmk}

To conclude the section we record the following 
computation which will be useful to us 
later on.

\begin{lem} \label{lem:boun-inc}
Let $A$ be a graded algebra. 
The right boundary map of the cone extension
\[
M_\infty A \to \Gamma A \to \Sigma A \tag{\ref{ext:cone} $\otimes A$}
\]
is $\delta_{\mathfrak K \otimes A} = -\Sigma \inc_1 \otimes A$.
\end{lem}
\begin{proof} In light of Remark \ref{rmk:right-boundary-tensor}, we may assume that $A = \ell$. In this case, the extension consists of ungraded algebras; 
by Theorem \ref{thm:triv-boundary}, we may thus prove the statement in $kk$ (in other words, 
we may assume $G$ to be the trivial group).
Denote the left and right boundary maps of the cone extension \eqref{ext:cone} by $\partial$ and $\delta$ respectively.
By Lemma \ref{lem:partial-s=e11} and the definition of $\delta$, 
\[
\delta = -\Sigma \partial \circ \Theta_Q = -\Sigma \inc_1 \circ (\Sigma \otimes \xi_L^{-1})  \circ (\Theta \otimes \Sigma) = -\Sigma \inc_1 \circ \cR_{\Sigma, \ell}(\xi_L^{-1}).
\]
To conclude we observe that $\xi_L^{-1} = \cL_{\Sigma, \ell}(\id_\Sigma) = \cR_{\Sigma, \ell}^{-1}(\id_\Sigma)$.
\end{proof}

\section{Graded infinity-sum algebras, cones and suspensions} \label{sec:inf-sum}

A \emph{graded $\ast$-algebra} is a graded algebra $R$ together 
with an involution $\ast \colon R \to R$ such that $R_g^\ast \subset R_{g^{-1}}$
for each $g \in G$. A graded \emph{sum $\ast$-algebra} is a graded $\ast$-algebra $R$ together with
homogeneous elements $x,y \in A_{1_G}$ such that 
\[
x^\ast x = y^\ast y = xx^\ast +yy^\ast = 1.
\]

If $x,y \in A_{1_G}$ make $A$ into a graded sum $\ast$-algebra, then $y^\ast x = 0$. This follows from left multiplying by $y^\ast$ and right multiplying by $x$ in the equality $xx^\ast + yy^\ast = 1$. Likewise we have that $x^\ast y = 0$. As a consequence, the assignment
\[
\isum \colon A \times A \to A, \qquad  a \isum b := xax^\ast + yby^\ast
\]
is a graded algebra homomorphism. Given graded $\ast$-algebra
homomorphisms $f,g \colon B \to A$, we write $f \isum g$
for the $\ast$-algebra homomorphism $b \mapsto f(b) \isum g(b)$.

A graded \emph{infinite-sum} 
$\ast$-algebra is a graded sum $\ast$-algebra $A$ together with a 
graded homomorphism $(-)^\infty \colon A \to A$ such that
$\isum \circ  (\id \times (-)^\infty) = (-)^\infty$, i.e. such that
\[
a \isum a^\infty = a^\infty \qquad (\forall a \in A).
\]

Our motivation for considering such algebras stems 
from the fact that they possess desirable properties in 
algebraic bivariant $K$-theory. Next,
we adapt some results from \cite{kk} to the graded setting.

\begin{prop}\label{prop:sum=sum}
If $B$ is a graded sum $\ast$-algebra
and $f, g \colon A \to B$ are graded algebra homomorphisms,
then $j(f) + j(g) \in \kkgr(A,B)$ equals $j(f\isum g)$.
\end{prop}
\begin{proof} Since $\kkgr$ is an additive category and 
$j$ an additive function (in the sense that it maps finite 
products to biproducts), it suffices to show that the codiagonal
map $\nabla \in \kkgr(B \times B, B)$ is equal to $j(\isum)$. 

By Proposition \ref{prop:incx=incy}, the inclusions $\iota_1, \iota_2 \colon B \to M_2B$ of $B$ in the top-left and bottom-right corners respectively are mapped to the same arrow in $\kkgr$. Thus, considering
the graded homomorphism
\[
\eps \colon (b_1,b_2) \in B \times B \mapsto \begin{pmatrix} b_1 & 0 \\ 0 & b_2
\end{pmatrix} \in M_2 B,
\] 
we have that $\nabla = j(\iota_1)^{-1} \circ j(\eps)$. In particular,
to prove the proposition  
it suffices to see that $j(\iota_1 \isum) = j(\eps)$.

Let $x,y \in B_{1_G}$ be the homogeneous elements 
that define the graded 
sum $\ast$-algebra structure on $B$. Put
\[
Q = \begin{pmatrix} x & y & 0\\
0 & 0 & x^\ast \\
0 & 0 & y^\ast
\end{pmatrix}
\]
for the matrix considered by Wagoner in \cite{wagoner}*{p. 355}, and set $u = Q^\ast = Q^{-1}$. Note 
that $u$ is a unitary element of $M_3 B$ which is homogeneous
of degree $1 \in G$.
By \cite{kk}*{Lemma 4.8.3}, we have 
\begin{equation}\label{ad:Q}
u\begin{pmatrix} b \isum b' & 0 & 0\\
0 & 0 & 0 \\
0 & 0 & 0
\end{pmatrix}u^\ast
= \begin{pmatrix} b & 0 & 0\\
0 & b' & 0 \\
0 & 0 & 0
\end{pmatrix}.
\end{equation}
In terms of the top-right corner inclusion $\jmath \colon M_2 B \to M_3 B$, equation \eqref{ad:Q} 
says that $\ad(u) \circ \jmath \circ \iota_1 \circ \isum = \jmath \circ \eps$. Applying the $j$ functor and using Proposition \ref{prop:ad},
we see $j(\jmath) \circ j(\iota_1 \isum) = j(\jmath) \circ j(\eps)$.
To conclude, we note that $j(\jmath)$ is an isomorphism by matricial stability.
\end{proof}

\begin{prop} \label{prop:inf=0}
If $A$ is a graded infinite-sum $\ast$-algebra
and $I \triangleleft A$ an ideal such that $I^\infty \subset I$,
then $I$ is $\kkgr$-equivalent to zero.
\end{prop}
\begin{proof} By Proposition \ref{prop:ad}, the matrix $u$
given in the proof of Proposition \ref{prop:sum=sum} determines 
a graded $\ast$-homomorphism $\ad(u) \colon M_3 I \to M_3 I$
representing the identity of $M_3 I$. Hence the same argument 
as in loc. cit. shows that the restriction 
$\isum' \colon  I \times I \to I$ of $\isum$ to $I$ is the codiagonal
map of $I$. Since $I^\infty \subset I$, we can also restrict
$(-)^\infty$ to a map $(-)^{\infty'} \colon I \to I$ satisfying
$(-)^{\infty'} = 1_I \isum' (-)^{\infty'}$. It follows that 
$j((-)^{\infty'}) = j(1_I \isum' (-)^{\infty'}) = j(1_I) + j((-)^{\infty'})$;
this goes to show that $j(1_I) = 0$ and therefore $j(I) = 0$.
\end{proof}

In the ungraded setting, an example of an infinite-sum $\ast$-algebra
is Karoubi's cone 
$\Gamma_X$ for any infinite set $X$ (\cite{arcor}*{Equation 2.2}, see also \cite{kk}*{Lemma 4.8.2}). We wish
to prove a similar statement for its graded analogue 
\[
\Gamma_X^\circ = \mathsf{span}_\ell\{f \in \Gamma_X : |x| |f(x,y)||y|^{-1} \text{ is constant whenever $f(x,y) \neq 0$}\}.
\]
Notice that this algebra is graded by setting $(\Gamma^\circ_X)_g = 
\mathsf{span}_\ell\{f \in \Gamma_X : |x| |f(x,y)||y|^{-1} = g \text{ if $f(x,y) \neq 0$.}\}$. It contains $M_X$ as a homogeneous ideal; the quotient $\Gamma_X^\circ/M_X$ is denoted by $\Sigma_X^\circ$. 

A non empty graded set $(X,d)$ with degree map $d \colon X \to G$ is said to be \emph{graded infinite} if for all $g \in G$ the set $X_g := d^{-1}(g)$ is either empty or infinite.

\begin{prop} \label{prop:grbij}
If $X$ is graded infinite, then there exists 
a graded bijection $X  \sqcup X \xto{\sim} X$.
\end{prop}
\begin{proof} If $X$ is graded infinite, then each component $X_g$
is either infinite or empty and thus there exist injections $\sigma_g, 
\tau_g \colon X_g \to X_g$ such that $\sigma_g \sqcup \tau_g \colon X_g \sqcup X_g \to X_g$ is a bijection. It follows that $\sigma = \sqcup_{g \in G} \sigma_g$ and $\tau = \sqcup_{g \in G} \tau_g$ assemble into the desired bijection.
\end{proof}

\begin{prop}[cf. \cite{kk}*{Lemma 4.8.2}] If $X$ is graded infinite, then $\Gamma_X^\circ$
is a graded infinite-sum $\ast$-algebra.
\end{prop}
\begin{proof} In view of Proposition \ref{prop:grbij},
we may consider a graded bijection $X \sqcup X \to X$
induced by graded injections $\sigma, \tau \colon X \to X$, 
with disjoint image, such that $X = \im(\tau) \sqcup \im(\sigma)$.
A direct verification shows that the elements
\[
u = \sum_{x \in X} \elmat_{\sigma(x),x}, \qquad v \in \sum_{x \in X} \elmat_{\tau(x),x}
\]
make $\Gamma_X^\circ$ into a graded sum $\ast$-algebra.

Next we will show that, for every $x,y \in X$,
\begin{equation}\label{inj-sigma-tau}
\tau^n(\sigma(x)) = \tau^m(\sigma(y)) \iff n = m, \quad x = y.
\end{equation}
Indeed, suppose without loss of generality that $n = m+k$
for some $k \geq 0$. By injectivity of $\tau^m$, we 
would have that $\tau^k(\sigma(x)) = \sigma(y)$. 
Since the images of $\tau$ and $\sigma$ are disjoint, 
it must be $k = 0$ and hence $n = m$. Finally, the injectivity of $\sigma$
lets us deduce that $x = y$.

From \eqref{inj-sigma-tau} we see that, for each $z \in \Gamma_X^\circ$, there is a well defined element of $\Gamma_X^\circ$ given by
\[
z^\infty := \sum_{n \ge 0} v^n u z u^\ast (v^n)^\ast
= \sum_{n \ge 0, x,y \in X} \elmat_{\tau^n(\sigma(x)),x} \cdot z \cdot \elmat_{y,\tau^n(\sigma(y))}
= \sum_{n \ge 0, x,y \in X} z(x,y)\cdot \elmat_{\tau^n(\sigma(x)),\tau^n(\sigma(y))}.
\]
By definition $z \mapsto z^\infty$ is an algebra homomorphism 
and makes $\Gamma_X^\circ$ into a graded infinite-sum 
$\ast$-algebra as desired.
\end{proof}

We now apply the definition of graded infinity sum $\ast$-algebra to 
a graded analogue of Karoubi's cone and the 
resulting suspension algebra.

\begin{coro} \label{coro:gamma=0}
If $X$ is a graded infinite set, 
then $\Gamma_X^\circ$ is $\kkgr$-equivalent to zero.
\end{coro}
\begin{proof} Apply Proposition \ref{prop:inf=0} to $I = A = \Gamma_X^\circ$.
\end{proof}

Note that if $(X,d)$ is any graded set, then 
\[
\gc{X} := X \times \N, \qquad d(x,n) = d(x)
\]
is graded infinite and there is a canonical inclusion 
$x \in X \mapsto (x,0) \in \gc{X}$.

\begin{defn} Let $X$ be a graded set such that that $X_{1_G} \neq \emptyset$. Define 
\[
\Gamma^{\gr}_X := \Gamma^\circ_{\gc{X}}, \qquad \Sigma^{\gr}_X := \Sigma^\circ_{\gc{X}}, \qquad M^{\gr}_X := M_{\gc{X}}.
\]
\end{defn}

\begin{coro} \label{coro:ungrad=Xhat}
Let $X$ be a graded set such that $X_{1_G} \neq \emptyset$ and let $x \in X_{1_G}$.
The graded inclusion $i_x \colon k \in \N \mapsto (x,k) \in \gc{X}$ induces 
algebra monomorphisms $M_\infty \hookrightarrow M^{\gr}_X$, $\Sigma \hookrightarrow \Sigma^{\gr}_X$, and $\Gamma \hookrightarrow \Gamma^{\gr}_X$
which are $\kkgr$-isomorphisms.
\end{coro}
\begin{proof} Since $\Gamma_\N^\circ = \Gamma_\N = \Gamma$ and $\Sigma_\N^\circ = \Sigma_\N = \Sigma$, the maps induced by the inclusion $i_x$ yield a diagram of cone extensions (and thus triangles) as follows:
\[
\begin{tikzcd}
M_\infty \arrow{r}\arrow{d} & \Gamma \arrow{r} \arrow{d} & \Sigma \arrow{d}\\
M_{\gc{X}} \arrow{r} & \Gamma^\circ_{\gc{X}} \arrow{r} & \Sigma^\circ_{\gc{X}}
\end{tikzcd}
\]
The leftmost vertical arrow is a $\kkgr$-equivalence by graded matricial stability. The vertical arrow in the middle is a $\kkgr$-equivalence because both its domain
and codomain are $\kkgr$-equivalent to zero. It follows, using that $\kkgr$ is a triangulated category,
that the rightmost vertical arrow is a $\kkgr$-equivalence.
\end{proof}

\subsection{Graded suspensions and deloopings}
A family of boundary maps that will be of interest to us 
come from graded cone extensions. Given a graded set $X$, we have an extension
\[
\Omega \Sigma^{\gr}_X \xto{\partial_X} M^{\gr}_X \to \Gamma^{\gr}_X \to \Sigma^{\gr}_X.
\]
Using that the inclusion $\inc^{\gr}_X \colon M_X \hookrightarrow M^{\gr}_X$ is
a $\kkgr$-isomorphism, we obtain a triangle:
\begin{equation}\label{ext:infcirc}
\Omega \Sigma^{\gr}_X \xto{(\inc^{\gr}_X)^{-1} \circ \partial_X} M_{X} \to \Gamma^{\gr}_X \to \Sigma^{\gr}_X.
\end{equation}
Since $\Gamma^{\gr}_X = 0$ in $\kkgr$ by Corollary \ref{coro:gamma=0}, 
it follows that the map $(\inc^{gr}_X)^{-1} \circ \partial_X$ is an isomorphism. In particular, by matricial stability, for any $x \in X_1$ we have
an isomorphism
\begin{equation}\label{def:partial-circ}
\partial_X^{\gr} := 
\Omega \Sigma^{\gr}_X \xto{(\inc^{\gr}_X)^{-1} \circ \partial_X} M_X \xto{j(\iota_x)^{-1}} \ell.
\end{equation}

We may describe $\partial_X^{\gr}$ more explicitly, in the same way as for
the ungraded cone extension.

\begin{prop} \label{prop:sLx}
Let $X$ be a graded infinite set such that $X_{1_G} \neq \emptyset$
and $x \in X_{1_G}$. Write $L_x$ for the class 
of $\sum_{i \ge 1} \elmat_{(x,i), (x,i+1)}$ in $\Sigma^{\gr}_X$
and $\xi_{L_x} \colon \ell \to \Omega \Sigma$ for the map corresponding 
to $[L_x] \in KH_1^{\gr}(\Sigma^{\gr}_X)$.
The following diagram commutes in $\kkgr$:
\[
\begin{tikzcd}
\ell \arrow{r}{\xi_{L_x}} \arrow{dr}[below]{\inc_{(x,1)} \qquad \quad }& \Omega \Sigma^{\gr}_X \arrow{d}{\partial^{\gr}_X}  \\ & M_X^{\gr}
\end{tikzcd}
\]
\end{prop}
\begin{proof} The morphism of extensions between 
the ungraded cone extension \eqref{ext:cone} and
extension \eqref{ext:infcirc} given by inclusions, as 
in Corollary \ref{coro:ungrad=Xhat}, extends 
to a morphism of triangles expressing 
$\partial^{\gr}_X$ in terms of $\partial$.
The conclusion 
now follows from Lemma \ref{lem:partial-s=e11}; we
leave the details to the reader.
\end{proof}

\begin{rmk} If $X = \ast$ with $|\ast| = 1_G$, then $\gc{X}$
is isomorphic to $\N$ with trivial grading. 
The inclusion $\inc_X^{\gr}$ corresponds to the inclusion $\iota_1 \colon \ell \to M_\infty$, and \eqref{ext:infcirc}
to the triangle $\ell \to \Gamma \to \Sigma$. 
Hence, the boundary \eqref{def:partial-circ} recovers
the isomorphism $\Omega\Sigma \xto{\partial} M_\infty \xto{\inc_1^{-1}} \ell$ and Proposition \ref{prop:sLx}
recovers \ref{lem:partial-s=e11} as a particular case.
\end{rmk}

\section{Units as morphisms}

The purpose of this section is to represent 
elements of $KH_1$ and $KH_1^{\gr}$ coming 
from units as certain arrows in the corresponding bivariant $K$-theory 
category. We first give a representation for units in ungraded algebras
as maps in $kk$. Next,
we use these results to deduce a representation in $\kkgr$
for homogeneous units of degree $1_G$ of strongly graded rings.

\subsection{Non-homogeneous units}

\begin{defn}\label{defn:cS}
Let $\cS := \ker(\ell[t,t^{-1}] \xto{\ev_1} \ell)$.
Note that a
homomorphism $\ell[t,t^{-1}] \to A$ corresponds to 
the choice of an idempotent $p:=\phi(1) \in A$ 
and a unit in $pAp$, namely $u:=\phi(t)$. We write
$\phi_{p,u}$ for such a homomorphism and $\nu_{p,u}$
for its restriction to $\cS$. If $A$ is unital and $u$ is a unit in $A$, we put
$\phi_u := \phi_{1,u}$ and $\nu_u := \nu_{1,u}$. 
\end{defn}

\begin{defn} \label{defn:unitmaps}
Let $A$ be a unital algebra. A unit $u \in A^\times$
determines a class $[u] \in K_1(A)$ which, via the canonical comparison map, 
determines an element in $KH_1(A)$ which we also call $[u]$.
We define $\xi_u \colon \ell \to \Omega $ to be 
the homomorphism corresponding to $[u] \in KH_1(A)$
via the isomorphism $KH_1(A) \simeq kk(\ell, \Omega A)$.
Since $kk(\ell, \Omega(-)) \simeq KH_1(-)$, it follows that for any 
non-necessarily unital map $f \colon R \to S$ between unital algebras we have:
\begin{equation}\label{eq:omega-xi=xi}
\Omega(j(f)) \circ \xi_u = \xi_{1-f(1)+f(u)}.
\end{equation}
In particular, if $A$ is a unital algebra then applying \eqref{eq:omega-xi=xi} to the unit $t \in \ell[t,t^{-1}]$ and any map $\phi_{p,u} \colon \ell[t,t^{-1}] \to A$ gives
\begin{equation}\label{Omegaphi-xi}
\Omega(j(\phi_{p,u})) \xi_t = \xi_{1+p-u}.
\end{equation}
\end{defn}

By \cite{kk}*{Section 4.10 and proof of Theorem 7.3.1} we know that $\nu_L \colon \cS \to \Sigma$ is an isomorphism. Upon tensoring by $\nu_L$, we 
obtain a natural isomorphism $\cS \otimes - \cong  \Sigma \otimes -$. This allows for the following definition.

\begin{defn} Put
\begin{equation}\label{map:wpnat}
\wp_{A,B} :=  \kkgr(\cS A, B)\xto{ (\nu_L^{-1} \otimes A)^\ast}  
kk(\Sigma A, B) \xto{\cU_{A,B}} kk(B, \Omega A).
\end{equation} 
and 
\begin{equation}\label{map:wp}
\wp := \wp_{\ell,\cS}(\id_\cS) = \Omega(\nu_L^{-1}) \xi_L.
\end{equation}
\end{defn}

The map \eqref{map:wpnat}
allows us to represent classes of units in $KH_1$ as classes of algebra homomorphisms in $kk$:

\begin{thm} \label{thm:rep-units}
Let $A$ be a unital algebra. The 
natural chain of isomorphisms
\[
kk(\cS, A) \xto{\wp_{\ell, A}} kk(\ell, \Omega A) \simeq KH_1(A)
\]
maps $j(\nu_{p,u})$
to the class of the unit $1-p+u \in A^\times$ in $KH_1(A)$.
\end{thm}
\begin{proof} Write 
$i \colon \cS \to \ell[t, t^{-1}]$ for the inclusion.
We want to show that $\xi_{1-p+u} = \Omega(\phi_{p,u}) \circ \xi_t$ agrees with
\[
\Omega(\nu_{p,u}) \Omega(j(\nu_L))^{-1} \xi_L =
\Omega(\phi_{p,u}) \Omega(i) \Omega(j(\nu_L))^{-1} \xi_L.
\]
Thus, it suffices to see that 
\[
\Omega(i) \Omega(j(\nu_L))^{-1} \xi_L = \xi_t.
\]
We claim that, to conclude, it suffices to see that $\xi_t$ factors through $\Omega(i)$.
Indeed, suppose that there exists a map $\zeta \colon \ell \to \Omega \cS$
such that $\xi_t = \Omega(i) \zeta$. 
Then 
\[
\xi_L = \Omega(j(\phi_L)) \xi_t = 
\Omega(j(\phi_L)) \Omega(i) \zeta = \Omega(j(\nu_L)) \zeta,
\]
and composing with $\Omega(i)\Omega(j(\nu_L))^{-1}$
on the left to both sides we obtain the desired equality.

Finally, we have to prove the existence of such a morphism $\zeta \colon \ell \to \Omega \cS$, which amounts 
to showing that $[t] \in KH_1(\ell[t, t^{-1}])$ lies in the image of $KH_1(i)$. It suffices to do so substituting 
$K_1$ for $KH_1$. 

Consider the elements $x = (t-1)$, $y = (t^{-1}-1)$ of $\cS$.  A direct computation 
shows that $xy+x+y = 0$, which says that in the unitalization $U$ of $\cS$ the element $x + 1_U$
is a unit with inverse $y + 1_U$. The 
map induced by $i$ on $K_1$ restricts 
to a map between the units of $U$ and those
of the unitalization $U'$ of $\ell[t,t^{-1}]$, 
which maps $x+ 1_U$ to $x+1_{U'}$.
To conclude, we note that the identification of $K_1(\ell[t,t^{-1}])$
with $\ker(K_1(U') \to K_1(\ell))$ maps $t$ to $x+1_{U'}$.
\end{proof}

\subsection{Graded units in strongly graded rings}

We view $\cS$ and $\ell[t,t^{-1}]$ as graded algebras 
via the trivial grading. A graded homomorphism 
$\ell[t,t^{-1}] \to \cS$ corresponds to a homogeneous 
idempotent $p \in A_1$ and a unit $u \in pA_1p$. 
We employ the same notation as in 
Definition \ref{defn:cS} 
for these homomorphisms and their restrictions to $\cS$.

\begin{prop} Let $C$ be a trivially graded algebra and 
$A$ a strongly graded algebra. There is an isomorphism
\[
\kkgr(C,A) \simeq kk(C, A_{1_G}) 
\]
which maps the class of a graded algebra homomorphism
$f \colon C \to A$ to the class of its 
corestriction $f| \colon C \to A_{1_g}$.
\end{prop}
\begin{proof} This follows directly from the proof of \cite{kkg}*{Theorem 6.1.4} (see also \cite{arcor}*{Remark 8.4}) 
and the bivariant version of Dade's teorem \cite{arcor}*{Theorem 10.1}.
\end{proof}

From Theorem 
\ref{thm:rep-units} and the proposition above, we obtain the main result of the section. 

\begin{thm} \label{thm:rep-grunits}
Let $A$ be a unital, strongly graded algebra, $p \in A_{1_G}$
an idempotent and $u$ a unit in $pA_{1_G}p$. Consider the map 
$\phi \colon \cS \to A$ given by $1 \mapsto p$, $t \mapsto u$.
Under the chain of isomorphisms
\[
\kkgr(\cS, A) \cong kk(\cS, A_{1_G}) \simeq KH_1(A_{1_G}),
\]
the arrow $j(\phi)$ has image $[1-p+u]$.
\qed
\end{thm}

For the next corollary we recall from \cite{arcor}*{Lemma 9.3} that given two arrows $\xi \in \kkgr(A,B)$, $\zeta \in \kkgr(C, D)$, 
their tensor product is defined as $\xi \otimes \zeta := (B \otimes \zeta) \circ (\xi \otimes C) = (\xi \otimes D) \circ (A \otimes \zeta)$.

\begin{coro} \label{lem:tensor-unit-idem}
Let $R$ and $B$ be unital graded algebras and $p \in R_{1_G}$, $q \in B_{1_G}$ two idempotents. Consider $u$ a unit of $R_{1_G}$ and $\inc_q \colon \ell \to B$
the algebra map sending $1$ to $q$. If $R$ is strongly graded and $u$ is a unit of $R_{1_G}$, then $\xi_{1 \otimes q, u \otimes q} = \xi_u \otimes \inc_q$.
\end{coro}
\begin{proof} The map $\inc_q$
defines a natural transformation $\id \Rightarrow - \otimes B$ and thus
we have the following commuting diagram:
\[
\begin{tikzcd}[column sep = large, row sep = large]
    \ell \arrow{r}{\sim} \arrow{r}\arrow{d}{\inc_q} &
    \Omega \Sigma \arrow{r}{\nu_L} \arrow{d}{\Omega \Sigma \otimes \inc_q}&
    \Omega \cS \arrow{r}{\Omega \phi_u} \arrow{d}{\cS \otimes \inc_q} & 
    \Omega R \arrow{d}{\Omega R \otimes \inc_q}\\
    B \arrow{r}{\sim} &
    \Omega \Sigma \otimes B \arrow{r}{\nu_L \otimes B} &
    \Omega \cS \otimes B \arrow{r}{\phi_u \otimes B} & 
    \Omega R \otimes B \\
\end{tikzcd}
\]
By Theorem \ref{thm:rep-grunits}, the top row corresponds 
to $\xi_u$. Composition by $\Omega R \otimes \inc_q$
corresponds to the morphism $KH_1^{\gr}(R) \to KH_1^{\gr}(R \otimes B)$ induced by $R \otimes \inc_q$; hence, the top row of the diagram followed
by the right-most vertical arrow corresponds to
$\xi_{(R \otimes \inc_q)(1), (R \otimes \inc_q)(u)} = \xi_{1 \otimes q, u \otimes q}$.
As the diagram shows, this has to coincide with $\inc_q = \ell \otimes \inc_q$ composed with $\xi_u \otimes B$, which is by definition $\xi_u \otimes \inc_q$.
\end{proof}

We conclude the section with some results on boundary maps 
from $K_1^{\gr}$to $K_0^{\gr}$.

\begin{prop} \label{prop:gr-index}
Let $\pi \colon R \to S$ be a surjective, 
graded $\ast$-algebra homomorphism; 
write $I := \ker(\pi)$. Assume that $R$ is strongly graded. 
Let $u \in S_1$ be a unit.
If $\hat u \in R_1$ is a partial isometry such that $\pi(\hat u) = u$, 
then the boundary map $\partial \colon K_1^{\gr}(R) \to K_0^{\gr}(I)$
maps $[u]$ to $[1-\hat u^\ast \hat u] - [1-\hat u \hat u^\ast]$.
\end{prop}
\begin{proof} In the case of (hermitian) ungraded $K$-theory, a 
more general result is proven in \cite{kkhlpa}*{proof of Lemma 11.1}; the former implies in particular that
$\partial' \colon K_1(R_{1_G}) \to K_0(I_{1_G})$
maps $[u]$ to $[1-\hat u^\ast \hat u] - [1-\hat u \hat u^\ast]$.
The conclusion follows from the comparison square
\[
\begin{tikzcd}
K_1(R_{1_G}) \arrow{d} 
\arrow{r}{\partial'} & K_0(I_{1_G}) \arrow{d}\\
K^{\gr}_1(R) \arrow{r}{\partial} & K^{\gr}_0(I)
\end{tikzcd}
\]
\end{proof}

\begin{lem}\label{lem:cS=sigma}
Let $X$ be a graded set such that $X_1 \neq \emptyset$. For any $x \in X_1$, the algebra map 
\begin{equation}\label{def:thetaX}
\theta_X \colon \cS \to \Sigma^{\gr}_X, \qquad  t \mapsto \sum_{i \ge 1} \elmat_{(x, i+1),(x, i)}    
\end{equation}
is a $\kkgr$-equivalence. 
\end{lem}
\begin{proof} By Corollary \ref{coro:ungrad=Xhat},
the map $\Sigma \to \Sigma^{\gr}_X$ induced
by the inclusion $\N \subset \{x\} \times \N$ is a $\kkgr$-equivalence. The result is thus implied by the fact that $\nu_L \colon \cS \to \Sigma$ is a $\kkgr$-equivalence.
\end{proof}

\section{Poincaré duality}

This section is devoted
to the proof of the graded analogue
of Poincaré duality \cite{kkhlpa}*{Theorem 11.2} and its consequences. 
Although not needed in the rest of this manuscript, we shall prove the result for any grading on $L(E)$ given by a \emph{weight function} $\omega \colon E^1 \to G$, that is, the one given by the extension of the rule $|v| = 1_G$, $|e| = \omega(e)$, $|e^\ast| = \omega(e)^{-1}$ for each $v \in E^0$, $e \in E^1$. We shall write $L_\omega(E)$ to emphasize that $L(E)$
is being considered as a graded algebra with grading induced by $\omega$. Recall that 
the \emph{dual graph} $E_t$ of a graph 
$E$ is given by vertex and edge sets
\[
E_t^0 = E^0, \qquad E_t^1 = \{e_t \colon e \in E^1\}
\]
and source and range functions
\[
r(e_t) = s(e), \qquad s(e_t) = r(e) \qquad (e \in E^1).
\]

\begin{thm} \label{thm:poinc}
If $E$ is a finite essential graph, and 
$\omega \colon E^1 \to G$ a weight function, 
then $- \otimes_\ell L_\omega(E)$ is 
left adjoint to $- \otimes \Omega L_\omega(E_t)$ as endofuntors of $\kkgr$. Thus,
for each $R, S \in \grAlg$ there are isomorphisms
\[
\kkgr(R \otimes_\ell L_\omega(E), S) \cong \kkgr(R,S\otimes_\ell \Omega L_\omega(E_t)). 
\]
natural in both $R$ and $S$.
\end{thm}
\begin{proof} We adapt the proof \cite{kkhlpa}*{Theorem 11.2} to the present setting, which we will
frequently cite in the argument below.
Consider $\mathcal P_{\ge 1}$ the set of paths of positive length; we endow this set with a grading via the weighted length function $e_1 \cdots e_n \mapsto \omega(e_1)\cdots \omega(e_n)$. From now on, we omit the weight $\omega$
from the notation. Given $v \in E^0$, 
the set of paths starting at $v$
will be denoted $\cP^v$. Those ending 
at $v$ will be denoted $\cP_v$. Both
are graded sets viewed as subsets of $\cP$.

Put
$X = \mathcal P_{\ge 1} \sqcup \{\bullet\}$ and 
set $|\bullet| = 1_G$. We shall view $L(E_t) \otimes L(E)$ and 
$L(E) \otimes L(E_t)$ as graded algebras via the 
tensor product grading.
The morphisms
\[
\rho_1(e) = \left[\sum_{\alpha \in \cP_{s(e)}} \elmat_{\alpha e, \alpha}\right], \qquad \rho_2(e_t) = \left[\sum_{\alpha \in \cP^{r(e)}} \elmat_{e\alpha, \alpha}\right]
\]
in loc. cit. can be correstricted to 
morphisms with codomain $\Sigma^\circ_X$
which we denote in the same way. 
Composing with the canonical 
map $\Sigma^\circ_X \to \Sigma^{\gr}_X$, we obtain a graded homomorphism 
$\rho  \colon L(E_t) \otimes L(E)  \to \Sigma^{\gr}_X$. Put
\begin{equation*}
\kappa \colon \Omega L(E_t) \otimes L(E) \to \ell, \qquad \kappa := \Omega L(E_t) \otimes L(E) \xto{\Omega \rho}
\Omega \Sigma^{\gr}_X \xto{\partial_X^{\gr}} \ell.
\end{equation*}
Tensoring by $\kappa$ on the right yields a natural map
\begin{align}\label{bij1}
\kkgr(R, S \otimes \Omega L(E_t)) \to \kkgr(R \otimes L(E), S), \qquad
\xi \mapsto  (S \otimes \kappa) \circ (\xi \otimes L(E)).
\end{align}
In the other direction, we consider the elements $u_1 = \sum_{e\in E^1} e \otimes e_t^\ast$ and $p = \sum_{v\in E^0} v \otimes v$ as in the ungraded case, noting that they lie in the homogenous component of degree zero of $L(E) \otimes L(E_t)$. Thus 
$u_1 = u+1-p$ is a degree zero unit of $L(E) \otimes L(E_t)$ and we have an induced map 
$\nu_{u_1} \colon \cS \to L(E) \otimes L(E_t)$. By Lemma \ref{thm:rep-units}, 
the map $\xi_{u_1} \in \kkgr(\ell, \Omega L(E) \otimes L(E_t))$ associated 
to $[u_1] \in KH_1^{\gr}(L(E) \otimes L(E_t))$ equals $\wp_{\cS, L(E) \otimes L(E_t)}(\nu_{u_1}) = \Omega(\nu_{u_1})\wp$.
We now consider the composition
\begin{align}
\nabla := \ell \xto{\wp} \Omega \cS \xto{\Omega \nu_{u_1}} \Omega \otimes L(E) \otimes L(E_t) \iso L(E) \otimes \Omega L(E_t).
\end{align}
This defines a natural map
\begin{align}\label{bij2}
\kkgr(R \otimes L(E), S)\to\kkgr(R, S \otimes \Omega L(E_t)), \qquad
\xi \mapsto  (\xi \otimes \Omega L(E_t)) \circ (R \otimes \nabla).
\end{align}

Similar to the ungraded case, to see that the compositions of \eqref{bij1} and \eqref{bij2} are bijections, it suffices to show that 
$(\kappa \otimes \Omega L(E_t)) \circ (\Omega L(E_t) \otimes \nabla)$ and $(L(E)\otimes \kappa) \circ (\nabla \otimes L(E))$ are isomorphisms in $\kkgr$. We will indicate how to adapt the argument for the first composition, the other one follows likewise.
Define $\zeta \colon \cS \otimes L(E_t) \to \Sigma^{\gr}_X \otimes L(E_t)$ to be
the restriction of
\[
\zeta' \colon \ell[t,t^{-1}] \otimes L(E_t) \to \Sigma^{\gr}_X \otimes L(E_t), \quad s \otimes 1 \mapsto (\rho_2 \otimes 1)(u), \quad 1 \otimes x \mapsto \rho_1(x) \otimes 1
\]
and consider the following 
permutations of tensor factors:
\begin{align*}
(2 4 3) &\colon \Omega L(E_t) \otimes \Omega \cS \iso \Omega \otimes \Omega \otimes \cS \otimes L(E_t); \\
(2 3) &\colon \Omega \otimes \Omega \otimes \Sigma^{\gr}_X \otimes L(E_t)
\iso \Omega \Sigma^{\gr}_X \otimes \Omega L(E_t).
\end{align*}
A direct calculation shows that 
$(\kappa \otimes \Omega L(E_t)) \circ (\Omega L(E_t) \otimes \nabla)$ 
agrees with the following composition:
\[
\Omega L(E_t) \xto{\Omega L(E_t) \otimes \wp} \Omega L(E_t) \otimes \Omega \cS \xto{(2 3) \circ (\Omega \otimes \Omega \otimes \zeta) \circ (2 4 3)} \Omega \Sigma^{\gr}_X \otimes \Omega L(E_t) \xto{\partial^{\gr}_X \otimes \Omega L(E_t)} \Omega L(E_t).
\]
Hence, to see that $(\kappa \otimes \Omega L(E_t)) \circ (\Omega L(E_t) \otimes \nabla)$ is a $\kkgr$-isomorphism it suffices to show that $\zeta$ is one. To this end we define, as in the ungraded case, the graded $\ast$-homomorphism
\[
\partial \colon L(E_t) \to \bigoplus_{v \in E^0} \Sigma^{\gr}_{P^v}, 
\qquad e_t \mapsto \sum_{\alpha \in \cP^v_{s(e)}} \elmat_{\alpha e, \alpha}.
\]
It lifts to a graded $\ast$-homomorphism $C(E_t) \to \bigoplus_{v \in E^0} \Gamma^\circ_{\gc{\cP^v}}$ which restricts to the canonical 
isomorphism $\cK(E_t) \simeq \bigoplus_{v \in E^0} M_{\cP^v}$. 
We thus have maps of triangles
\begin{equation}
\begin{tikzcd}\label{diag:delta}
\ell^{E^0} \arrow{d}{\sim} \arrow{r} & \ell^{E^0} \arrow{r}\arrow{d}[left]{\inc}\arrow{d}{\sim} & L(E_t) \arrow[equals]{d} \arrow{r} & \Sigma \ell^{E^0}\arrow{d}{\sim} \\
\cK(E_t) \arrow{d}{\sim} \arrow{r} & C(E_t) \arrow{r}\arrow{d} & L(E_t) \arrow{d}{\partial} \arrow{r} & \Sigma \cK(E_t)\arrow{d}{\sim} \\
\bigoplus_{v \in E^0} M_{\cP^v} \arrow{r} & 
\bigoplus_{v \in E^0} \Gamma^{\gr}_{\cP^v} \arrow{r}& 
\bigoplus_{v \in E^0} \Sigma^{\gr}_{P^v} 
\arrow{r}{\sim} & \Sigma \bigoplus_{v \in E^0} M_{\cP^v}
\end{tikzcd}
\end{equation}

Hence, the boundary map $L(E_t) \to \Sigma \cK(E)$
corresponds to the $\kkgr$-class of $\partial$ and 
have a triangle
\[
\ell^{E^0} \xto{\inc} L(E_t) \xto{\partial} \bigoplus_{v \in E^0} \Sigma^{\gr}_{\cP^v}.
\]
Tensoring with $\cS$ and $\Sigma^{\gr}_X$ respectively, 
we obtain two distinguished triangles in $\kkgr$. 
To conclude the proof that $\zeta$ 
is an isomorphism, we will complete
$\zeta$ to a morphism of triangles as in the following diagram, where
both dashed arrows will be isomorphisms.

\begin{equation}\label{diag:upsilons}
\begin{tikzcd}
\cS \otimes \ell^{E^0} \arrow{r}{\cS \otimes \inc} \arrow[dashed]{d}[left]{\Upsilon_1}& \cS \otimes L(E_t) \arrow{d}{\zeta} \arrow{r}{\cS \otimes \partial} & \cS \otimes \bigoplus_{v \in E^0} \Sigma^{\gr}_{\cP^v} \arrow[dashed]{d}{\Upsilon_2} \\
\Sigma^{\gr}_X \otimes \ell^{E^0} \arrow{r}{\Sigma^{\gr}_X\otimes\inc} & \Sigma^{\gr}_X \otimes L(E_t) \arrow{r}{\Sigma^{\gr}_X\otimes\partial} & \Sigma^{\gr}_X \otimes \bigoplus_{v \in E^0} \Sigma^{\gr}_{\cP^v}.
\end{tikzcd}
\end{equation}

We construct the left hand arrow first using
the map $\theta_{X}$ as defined in 
\eqref{def:thetaX}. Set $\Upsilon_1 = \theta_{X} \otimes \ell^{E^0}$, which is an isomorphism
by Lemma \ref{lem:cS=sigma}. We shall
now see that $\zeta (S \otimes \inc) = (\Sigma^\circ_{\gc{X}} \otimes \inc) \circ \Upsilon_1$.
By additivity, it suffices to see that these compositions agree in each factor $\cS \otimes v$. In view 
of Lemma \ref{thm:rep-units}, this boils down 
to checking whether $1 - \theta_{X}(1) \otimes v + \theta_{X}(t) \otimes v$
and $1-\zeta(1 \otimes v) + \zeta(t \otimes v)$ represent the same 
class in $KH_1^{\gr}(\Sigma^{\gr}_X \otimes L(E_t))$. As in the ungraded setting, this follows from a direct computation using Proposition \ref{prop:gr-index} and the fact that, in this particular case, the boundary map $\partial \colon KH_1^{\gr}(\Sigma^{\gr}_X \otimes L(E_t)) \to KH_0^{\gr}(L(E_t))$ is an isomorphism.

Now we turn to defining the dashed right-most arrow. Write $\tau := \ell\{t,t^\ast : t^\ast t = 1\}$ where $|t| = 1_G$. Recall that there is an isomorphism $M_\infty \cong \ker(\tau \hookrightarrow \ell[t,t^{-1}])$ mapping $\elmat_{1,1}$ to $1-tt^\ast$.
As in the ungraded case, the restriction of $\zeta'$ to $\ell[t,t^{-1}] \otimes 1 \subset \ell[t,t^{-1}] \otimes L(E_t)$ can be extended to a graded homomorphism $\hat \zeta \colon \tau \to \Sigma^{\gr}_X \otimes L(E_t)$. 
Put $\tau_0 := \ker(\tau \xto{\ev_1} \ell)$. Consider now the following morphisms of triangles:
\begin{equation}\label{diag:eta}
\begin{tikzcd}
M_\infty L(E_t) \arrow{r}{} \arrow[equals]{d} & \tau_0 \otimes L(E_t) \arrow{d} \arrow{r} & \cS \otimes L(E_t) \arrow{d}{i \otimes L(E_t)} \arrow{r}[above]{d}
\arrow{r}[below]{\sim} & \Sigma M_\infty L(E_t) \arrow[equals]{d} \\
M_\infty L(E_t) \arrow{r}{} \arrow{d}{\hat \zeta|} & \tau \otimes L(E_t) \arrow{d}{\hat \zeta} \arrow{r} & \ell[t,t^{-1}] \otimes L(E_t) \arrow{d}{\zeta'} \arrow{r} & \Sigma M_\infty L(E_t) \arrow{d}{\Sigma(\hat \zeta|)}
\\
\Sigma^{\gr}_X \mathcal K(E_t) \arrow{r} & 
\Sigma^{\gr}_X C(E_t) \arrow{r} &
\Sigma^{\gr}_X L(E_t) \arrow{r} & 
\Sigma\Sigma^{\gr}_X \cK(E_t)
\end{tikzcd}    
\end{equation}
Since $\tau_0$ is trivially graded and 
by \cite{kk}*{Lemma 7.3.2} it is $kk$-equivalent to zero as an ungraded algebra, it follows that 
$\tau_0 \cong 0$ in $\kkgr$. In particular, the right boundary map $d$ of the top triangle 
is an isomorphism. Together with \eqref{diag:delta}, 
the diagram above says in particular that we have a 
commuting diagram as follows:
\begin{equation*}
\begin{tikzcd}
\cS \otimes L(E_t) \arrow{d}{\zeta} \arrow{r}[above]{d}
\arrow{r}[below]{\sim} & \Sigma M_\infty L(E_t) 
\arrow{d}{\Sigma(\hat\zeta|)} & \\
\Sigma^{\gr}_X  L(E_t) \arrow{r} \arrow{d}{\Sigma^{\gr}_X (\partial)} & 
\Sigma \Sigma^{\gr}_X  \cK(E_t) \arrow{r}{\sim} & \Sigma^{\gr}_X  \Sigma \cK(E_t) \arrow{d}{\sim}  \\
\Sigma^{\gr}_X   \bigoplus_{v \in E^0} \Sigma^{\gr}_{\cP^v} \arrow{rr}{\sim} & &
\Sigma^{\gr}_X  \Sigma \bigoplus_{v \in E^0} M_{\cP^v}
\end{tikzcd}
\end{equation*}
From this we obtain an isomorphism $\mu\colon
\Sigma \Sigma^{\gr}_X  \cK(E_t) \to \Sigma^{\gr}_X  \bigoplus_{v\in E^0} \Sigma^{\gr}_{P^v}$ such that 
\begin{equation}
    \Sigma^{\gr}_X (\partial) \circ \zeta = \mu \circ \Sigma(\hat\zeta|) \circ d
\end{equation}
Similary, by the same argument as in the ungraded case we obtain a commuting diagram
\[
\begin{tikzcd}
 L(E_t)  \arrow{r}{\partial}
 \arrow{d}[left]{\inc_1 \otimes L(E_t)} 
 \arrow{d}[right]{\sim} & \bigoplus_{v \in E^0} \Sigma^{\gr}_{\cP^v} \arrow{r}[above]{\bigoplus_v \inc_v}  
 \arrow{r}[below]{\sim} & \bigoplus_{v \in E^0} \Sigma^{\gr}_{X}  \arrow{d}[left]{\sim} \arrow{d}{\sum_{v \in E^0} \Sigma^{\gr}_X  \otimes q_v }\\
 M_\infty L(E_t) \arrow{rr}{\hat\zeta|} &  & \Sigma^{\gr}_X \cK(E_t)
\end{tikzcd}
\]
guaranteeing the existence 
of an isomorphism $\mu' \colon \bigoplus_{v \in E^0} \Sigma^{\gr}_{\cP^v}
\to \Sigma^{\gr}_X \cK(E_t)$ such that 
\[
\mu' \circ \Sigma(\partial) \circ \Sigma(\inc_1 \otimes L(E_t))^{-1} = \Sigma(\hat\zeta|).
\]
Further, as we have an isomorphism $\nu_L \colon \cS \iso \Sigma$, 
it follows that
\[
(\nu_L \otimes \bigoplus_{v \in E^0} \Sigma^{\gr}_{\cP^v}) \circ (\cS \otimes \partial) = 
\Sigma(\partial) \circ (\nu_L \otimes L(E_t))
\]
Thus, setting $\mu'' = 
\nu_L \otimes \bigoplus_{v \in E^0} \Sigma^{\gr}_{\cP^v}$, 
we get
\[
\Sigma^{\gr}_X \circ \zeta = \mu \mu' \mu'' \circ (\cS \otimes \partial) \circ 
 (\nu_L \otimes L(E_t))^{-1} \circ (\Sigma\inc_1 \otimes L(E_t))^{-1} \circ d.
\]

To conclude, we will see that $(\nu_L \otimes L(E_t))^{-1} \circ (\Sigma\inc_1 \otimes L(E_t))^{-1} \circ d = -\id$. 
Indeed, by Lemma \ref{lem:boun-inc} we have a commuting diagram
\[
\begin{tikzcd}
    M_\infty L(E_t) \arrow[equals]{d} \arrow{r} & \tau_0 L(E_t)
    \arrow{d} \arrow{r} & \cS L(E_t)
    \arrow{r}{d} \arrow{d}[left]{\nu_L \otimes L(E_t)} & \Sigma M_\infty L(E_t) \arrow[equals]{d} \\
    M_\infty L(E_t) \arrow{r} & \Gamma L(E_t) \arrow{r} & \Sigma L(E_t)
    \arrow{r}{\delta} & \Sigma M_\infty L(E_t),
\end{tikzcd}
\]
where $\delta = -(\Sigma \inc_1 \otimes L(E_t))$. Therefore
$d = -(\Sigma \inc_1 \otimes L(E_t)) \circ (\nu_L \otimes L(E_t))$ and 
\[
\Sigma^{\gr}_X \circ \zeta = -\mu\mu'\mu'' \circ (\cS \otimes \partial).
\]
We may thus complete
\eqref{diag:upsilons}
by setting $\Upsilon_2 = -\mu\mu'\mu''$. This finishes the proof.
\end{proof}

\begin{coro} \label{coro:uf}
Let $E$ and $F$ be finite graphs
with $E$ essential. If $f \colon L(E) \to L(F)$
is a graded algebra homomorphism, then the 
chain of isomorphisms
\[
\kkgr(L(E), L(F)) \xto{\eqref{bij2}} \kkgr(\ell, L(F) \otimes \Omega L(E_t)) 
\xto{\eqref{thm:rep-grunits}} KH_1((L(F) \otimes L(E_t))_0).
\]
maps $j(f)$ to the class of the unit
\begin{equation}\label{def:uf}
u_f := 1 \otimes 1 - \sum_{v \in E^0} f(v) \otimes v + \sum_{e \in E^1} f(e) \otimes e_t^\ast.
\end{equation}
\end{coro}
\begin{proof} The map in question 
is given by tensoring $\xi \in \kkgr(L(E), L(F))$
by $\Omega L(E)$, precomposing by 
\[
\ell \xto{u_1} \Omega(L(E) \otimes L(E_t))
\iso L(E) \otimes \Omega L(E_t)
\]
and postcomposing again with the inverse of the isomorphism $\Omega(L(E) \otimes L(E_t)) \iso L(E) \otimes \Omega L(E_t)$.
This coincides with the composition $(\Omega f \otimes L(E_t)) \circ u_1$; that is, 
the map corresponding to the image of $[u_1] \in KH_1((L(E) \otimes L(E_t))$
under $KH_1(f \otimes L(E_t))$. It remains to note that 
$u_f=(f \otimes L(E_t))(u_1)$.
\end{proof}

\begin{conv}
For the rest of the article, we will assume that $G = \Z$; in 
particular, we will use additive notation for the sum of degrees of 
homogeneous elements.
\end{conv}

\subsection{Graded Morita invariance and source elimination}

We record some observations on how one can extend Poincaré duality to non-necessarily essential graphs. We first recall
the notions of full idempotents and 
source elimination (\cite{flowinv}*{Definition 1.2}).

An idempotent $p$ of a unital ring
$R$ is \emph{full} if $RpR = R$, that is, 
if there exists $n \in \N$ and
$x_1, \ldots, x_n, y_1, \ldots, y_n \in R$ such that
\[
\sum_{i \in \N} y_i p x_i = 1.
\]
Notice that any unital ring homomorphism maps 
full idempotents to full idempotents.
Let $E$ be a graph and $v \in \sour(E) \setminus \sink(E)$. The
\emph{source elimination graph} $E_{\setminus v}$ is given by 
\[
E_{\setminus v}^0 = E^0 \setminus\{v\}, \quad E_{\setminus v}^1 = E^1 \setminus s^{-1}(v), \qquad
r_{E\setminus v} = r, \quad s_{E\setminus v} = s|_{E^0\setminus\{v\}}.
\]
By \cite{kkhlpa}*{Lemma 8.3}, the element $p = 1-v$ is a full homogeneous idempotent of $L(E)$ and 
the image of the graph inclusion induced map $\inc_v \colon L(E_{\setminus v}) \to L(E)$
is exactly $p L(E) p$. 
As noted in \cite{haz-vnreg}*{p. 230}, when $\ell$ is a field
source elimination preserves the (graded) Morita equivalence class of a Leavitt path algebra; 
this stems from the fact that, by the graded uniqueness theorem, the map $\inc_v$ is injective and 
thus $L(E_{\setminus v}) \cong pL(E)p$.
In this direction, we wish to prove that $\inc_v$ is a $\kkgr$-isomorphism. This is implied by the result below.

\begin{prop} \label{prop:gr-fullcor}
Let $R$ be a graded algebra. If $p \in R$ a homogeneous full idempotent of degree zero, 
then the inclusion $pRp \subset R$ is a $\kkgr$-isomorphism.
\end{prop}
\begin{proof} 
We adapt \cite{kkhlpa}*{Lemma 8.12}.
Let $x_1, \ldots, x_n, y_1, \ldots, y_n \in R$ 
be such that $1 = y_1px_1 + \cdots y_n p x_n$. Taking degree zero components at both
sides of this equality, enlarging $n$ if necessary 
we may assume that all $x_i, y_j$ are homogeneous such that $|x_i| = - |y_i|$. 
Substituting $x_i$ by $px_i$ and $x_i^\ast$ by $y_ip$ if necessary, 
we may also assume that $x_i \in pR$ and $y_i \in Rp$.
Put $d_i = |y_i|$. For the rest of the proof, we shall consider
the grading on $M_n$ given by the assingment $i \mapsto d_i$. 
Consider the elements
\[
c = \sum_{j = 1}^n \elmat_{j,1} x_j \in M_n pR, \qquad r = \sum_{j=1}^n \elmat_{1,j} y_j \in M_n Rp
\]
and notice that these elements are homogeneous and that $|c||r|=1$
and that $c \elmat_{1,1}M_nR\elmat_{1,1} r \subset M_n pRp$. Further, 
sice $rc = \elmat_{1,1}$, it follows that $wrcw' = ww'$
for each $w,w'\in \elmat_{1,1}M_nR\elmat_{1,1}$. We thus have 
a well-defined graded homomorphism
\[
\ad(c,r) \colon \elmat_{1,1}M_nR\elmat_{1,1} \to M_n pRp, \qquad w \mapsto cwr,
\]
and, by Proposition \ref{prop:ad}, upon composing with the inclusion 
$M_n(\inc_p) \colon M_n pRp \hookrightarrow M_n R$ it conicides in $\kkgr$ with the inclusion $\elmat_{1,1}M_nR\elmat_{1,1} \hookrightarrow M_n R$. 
Therefore, if we define 
\[
\phi \colon R \cong \elmat_{1,1}M_nR\elmat_{1,1} \xto{ad(c,r)} M_n pRp,
\]
it satisfies $j(M_n(\inc_p)\phi) = j(\iota_1^R)$. In particular $M_n(\inc_p)\phi$
is a $\kkgr$-isomorphism. 
A similar argument applied to $cp, rp \in M_n pRp$ says that the composition
\[
\phi \inc_p = pRp \cong \elmat_{1,1}M_n pRp\elmat_{1,1} \xto{\ad(cp,pr)} M_n pRp
\]
agrees in $\kkgr$ with $\iota_1^{pRp}$; hence $\phi \inc_p$ is also a $\kkgr$-isomorphism. Finally, this says that 
$\phi$ is a $\kkgr$-isomorphism which, in turn, proves that 
$j(\inc_p) = j(\phi)^{-1}j(\iota_1^{pRp})$ is an isomorphism as desired.
\end{proof}

\begin{coro} \label{coro:s-elim-kkg}
Assume that $\ell$ is a field. 
If $E$ is a graph with at least two vertices and
$v \in \sour(E) \setminus \sink(E)$, then the inclusion $L(E_{\setminus v}) \to L(E)$
is a $\kkgr$-isomorphism.
\qed
\end{coro}

\begin{rmk} By \cite{gramor}*{Theorem 5.3}, two graded unital rings
are graded Morita equivalent if and only if there exists a graded set structure
on $\N$, say $X = (\N, d), d \colon \N \to G$, such that $M_X S \cong M_\infty R$
as graded algebras. (As per our conventions, here $M_\infty R$ means $M_\N R$ where $\N$
is equipped with the trivial grading). In particular, this is a way to shows that $G$-stable functors are Morita invariant.
\end{rmk}

The results of this section say that given a regular graph $E$, upon finitely many source eliminations we may find an essential graph $F$ such that we have a 
$\kkgr$-isomorphism $L(F) \to L(E)$. 
Theorem \ref{thm:poinc} then implies that tensoring by $L(E)$ is left adjoint 
to tensoring by $\Omega L(F_t)$.

\section{The relationship between \topdf{$\kkgr$}{kkgr}-maps and graded algebra maps} \label{sec:kkgr-to-alg}

Consider $E$ a finite graph and the Cohn extension
\eqref{ext:cohn}. Write $\partial_E$ and $\delta_E$ for its left and right boundary. By \cite{arcor}*{Corollary 11.9}, in $\kkgr$ we have a triangle
\begin{equation}\label{triang:cohn}
\ell^{\reg(E)} \xto{I-\sigma A_E^t} \ell^{E^0} \xto{\inc} L(E).
\end{equation}
In particular, for a given graded algebra $R$, applying $\kkgr(-, R)$
to \eqref{triang:cohn} yields an exact sequence
\[
\begin{tikzcd}
 \kkgr(\Sigma^{E^0}, R) \arrow{r}{(\Sigma(I-\sigma A_E^t))^\ast} &
\kkgr(\Sigma^{\reg(E)}, R) \arrow[out=0, in=180]{d}{\delta_E^\ast} & &\\
& \kkgr(L(E), R) \arrow[out=0, in=180]{d}{\inc^\ast} &
\\ & \kkgr(\ell^{E^0}, R) \arrow{r}{(I-\sigma A_E^t)^\ast}& 
\kkgr(\ell^{reg(E)}, R).  
\end{tikzcd}
\]

From this sequence we can obtain, by taking appopriate 
kernels and cokernels, a short exact sequence 
involving $\kkgr(L(E),R)$. This 
is what in the ungraded case is referred to as 
the Universal Coefficient Theorem (UCT). Recall 
that the \emph{dual Bowen-Franks module} of $E$
is $\gBF^\vee(E) = \coker(I^t - \sigma A_E)$; 
when $E$ is essential, it follows that $\gBF(E_t) = \gBF^\vee(E)$.
With this notation in place, we state a theorem which, in particular, contains a graded version of the UCT.

\begin{thm}[UCT]\label{thm:uct}
Let $E$ be a finite graph and $R$ a graded algebra. There is a diagram
with exact top-row
\[
\begin{small}
\begin{tikzcd}
    0 \arrow{r} & \gBF^\vee(E) \otimes_{\Z[\sigma]} KH_1^{\gr}(R) \arrow{r}{d} & \kkgr(L(E), R) \arrow{r}{\ev}
    & \hom_{\Z[\sigma]}(\gBF(E), KH^{\gr}_0(R)) \arrow{r} & 0\\
    & & \text{$[L(E), R]$} \arrow{u}{\overline j} \arrow[bend right = 30]{ru}{\can^\ast \circ KH^{\gr}_0} & &
\end{tikzcd}    
\end{small}
\]
such that:
\begin{enumerate}[i)]
    \item The map $\overline j$ is the factorization of $j$ through the category 
    of graded $\ell$-algebras with graded homorphisms up to polynomial homotopy.
    \item The map $\ev$ corresponds to the assignment between hom-sets 
    of the functor $\kkgr(\ell, -) = KH_0^{\gr}$, followed by precomposition by the
    canonical map $\can \colon \gBF(E) \to KH_0^{\gr}(L(E))$.
    \item  The map $d$ is obtained from the right boundary map $\delta_E \colon L(E) \to \Sigma^{\reg(E)}$, by passing the composition
    \[
        \Z[\sigma]^{\reg(E)} \otimes_{\Z[\sigma]} KH^{\gr}_1(R)
        \iso \kkgr(\Sigma^{\reg(E)}, L(F)) \xto{\delta^\ast} \kkgr(L(E), L(F))
    \]
    to the quotient module $\gBF^\vee(E) \otimes_{\Z[\sigma]} KH_1^{\gr}(R)$. 
    \item If $E$ is an essential graph, 
    then for any $v \in E^0$ and unit $z \in R$ represented by a map $\xi_z \colon \ell \to \Omega R$, the isomorphism $\kkgr(L(E), R) \cong \kkgr(\ell,\Omega(R \otimes L(E_t))
    \cong KH_1^{\gr}(R \otimes L(E_t))$ maps $d(v \otimes z)$ to $[1 \otimes 1 - 1\otimes v + z \otimes v]$.
\end{enumerate}
\end{thm}
\begin{proof} We prove iv), the other assertions are proved in the same way as in ungraded case; see e.g. \cite{kklpa1}*{Corollary 7.20} and \cite{kkhlpa}*{Theorem 12.1}.
Assume that $E$ is essential. Recall that $\xi_z$ corresponds to an arrow 
$\cV_{\ell, R}(\xi_z) \colon \Sigma \to R$ via the assignment \eqref{def:cV}.
Thus, if $p_v \colon \ell^{E^0} \to \ell$ denotes the projection
to the $v$-th coordinate, then $d(v \otimes \xi_z)$ coincides with
the composition
\[
L(E) \xto{\delta} \Sigma^{E^0} \xto{\Sigma p_v} \Sigma \xto{\cV_{\ell, R}(\xi_z)} R.
\]
Recall also that the arrow above is assigned to an element 
$\kkgr(\ell,\Omega(R \otimes L(E_t))$ in the following way:
first, one tensors by $L(E_t)$; next one applies the loop functor $\Omega$ and,
at last, one precomposes by the arrow 
$\xi_{u_1} \colon \ell \to \Omega(L(E) \otimes L(E_t))$
given by the degree zero unit $u_1 \in L(E) \otimes L(E_t)$. 
Call this element $\eta := \Omega(d(v \otimes \xi_z) \otimes L(E_t)) \circ
\xi_{u_1}$.

Notice that, by Remark \ref{rmk:right-boundary-tensor}
and the definition of \eqref{def:cV}, we have
$\delta_E \otimes L(E_t) = \delta_{\cC(E) \otimes L(E_t)}$
and $\cV_{\ell, R}(\xi_z) \otimes L(E_t) = \cV_{\ell, R\otimes L(E_t)}(\xi_z \otimes L(E_t))$. We shall drop the subscripts under $\cV$ 
to ease the notation. 
Consider now the following diagram and note that 
the composition of the top row agrees with $\eta$: 
\[
\begin{small}
\begin{tikzcd}[row sep = large, column sep = large]
\ell \arrow{r}{\xi_{u_1}} \arrow[equals]{d} &
\Omega L(E) \otimes L(E_t) \arrow[equals]{d} \arrow{r}{\Omega \delta_{\cC(E)\otimes L(E_t)}} 
& \Omega\Sigma \otimes \ell^{E^0} \otimes L(E_t) \arrow{r}{\Omega \Sigma p_v \otimes L(E_t)} & \Omega\Sigma \otimes L(E_t) 
\arrow{r}{\Omega \cV(\xi_z \otimes L(E_t))} &  
\Omega R \otimes L(E_t)\\
\ell \arrow{r}{\xi_{u_1} } &
\Omega(L(E) \otimes L(E_t)) \arrow{r}{\partial_{\cC(E)\otimes L(E_t)}} 
& \ell^{E^0} \otimes L(E_t) \arrow{u}{\xi_L \otimes \ell^{E^0} \otimes L(E_t)} \arrow{r}{p_v \otimes L(E_t)} 
& L(E_t) \arrow{u}{\xi_L \otimes L(E_t)}\arrow{r}{\xi_z \otimes L(E_t)} &  
\Omega R \otimes L(E_t) \arrow[equals]{u}
\end{tikzcd}
\end{small}
\]

As in the ungraded case (\cite{kkhlpa}*{proof of Lemma 12.3}), one checks that
the boundary of $u_1$ is the class of $\sum_{v\in E^0}{\chi_v \otimes v}$,
and thus
\[
\eta = \xi_z \otimes L(E_t) \circ p_v \circ \left(\sum_{v \in E^0} \chi_v \otimes v\right)
= \xi_z \otimes \inc_v.
\]
Finally, by Lemma \ref{lem:tensor-unit-idem} $\xi_u \otimes \inc_v = \xi_{1 \otimes v, z \otimes v}$ which 
corresponds to the class in $KH_1^{\gr}(R \otimes L(E_t))$ of the unit $1 \otimes 1 - 1 \otimes v + z \otimes v$ as desired.
\end{proof}

We now want to use the UCT to investigate the relationship
between graded algebra maps between Leavitt path algebras
and morphisms between them in $\kkgr$. 
First, we set some conventions. 

\begin{defn}\label{defn:primitive}
A regular graph $E$ is \emph{primitive} if its 
adjacency matrix is primitive
(\cite{coding}*{Definition 4.5.7 and Theorem 4.5.8}), that is, if
there exists $N \ge 1$  such that $(A_E^N)_{v,w} > 0$
for all $v,w \in E^0$. 
\end{defn}

\begin{rmk} If $E$ is a primitive graph, 
then by definition there exists
$N \ge 1$ such that there is a path 
of length $N$ between each pair of vertices.
In particular, primitive graphs are essential. 
\end{rmk}

Our interest for primitive graphs stems
from the following.

\begin{prop} \label{prop:idem-primi} Assume that $\ell$ is a field. If $E$ is a primitive graph, then
for each $e \in E^1$ the idempotent $ee^\ast \in L(E)_0$
is full as an element of $L(E)_0$.
\end{prop}
\begin{proof} 
Recall that if we put
\[
L(E)_{0,n} = \mathrm{span}_\ell\{\alpha\beta^\ast : r(\alpha) = r(\beta), |\alpha| = |\beta| = n\}
\]
for each $n \ge 0$, then $L(E)_0 = \bigcup_{n \ge 0} L(E)_{0,n}$.
Recall also that, writing $\cP_{v,n}$ for the set of paths 
of length $n$ ending at a vertex $v$, there are isomorphisms
\begin{equation}\label{iso:L0n=matr}
L(E)_{0,n} \cong \bigoplus_{v \in E^0} M_{\cP_{v,n}}, \qquad \alpha\beta^\ast \mapsto \elmat_{\alpha, \beta} \in M_{\cP_{r(\alpha),n}}
\end{equation}
for each $n \ge 0$.

Since $L(E)$ is an inreasing union of its unital subalgebras $L(E)_{0,n}$, 
to see that $ee^\ast \in L(E)_0$
is full it suffices to see that it is so in $L(E)_{0,n}$ for some $n \in \N$.
Let $N \ge 1$ be such that 
there exists a path of length $N$ between 
every pair of vertices. Observe that
\[
ee^\ast = \sum_{s(\alpha) = r(e), |\alpha| = N} e\alpha(e\alpha)^\ast
= \sum_{v\in E^0} \sum_{s(\alpha) = r(e), |\alpha| = N, r(\alpha) =v} e\alpha(e\alpha)^\ast.
\]
Thus, under 
the isomorphism \eqref{iso:L0n=matr} applied to $n = N+1$, the idempotent $ee^\ast$
is mapped to a sum of diagonal matrices
\[
\sum_{v\in E^0} \sum_{s(\alpha) = r(e), |\alpha| = N, r(\alpha) =v} \elmat_{e\alpha,e\alpha}.
\]
Given that matrix rings over a field are simple algebras, to conclude
it suffices to prove that the coordinate of element above corresponding
to each algebra $M_{\cP_{v, N+1}}$ is non-zero. This amounts 
to showing that for each set $\{e\alpha : r(\alpha) = v, |\alpha| = N\}$ 
is non-empty,  which is implied by the fact that $(A_E^N)_{r(e),v} >0$
for all $v \in E^0$.
\end{proof}

\begin{rmk} In the proof of Proposition \ref{prop:idem-primi}, 
we only need that 
for each vertex $v$ in the graph $E$ there exists
some $N_v \ge 1$ such that the $v$-th row of $A_E^N$
has positive entires. However, if $E$ is essential,
this condition is equivalent to $E$ being primitive. 
Indeed, put $N = \max_{u \in E^0} N_u$
and fix $v,w \in E^0$. Since $E$ is essential, 
inductively we may find a path $\beta$ 
of length $N-N_w$ ending at $w$. 
By hypothesis we also have a path $\alpha$
of length $N_{s(\beta)}$ from $v$ to $s(\beta)$; hence $\alpha\beta$
is a path of length $N$ from $v$ to $w$. 
\end{rmk}

\begin{conv}\label{conv:regufield}
From now on, we shall assume that $\ell$
is a field and all graphs considered are primitive.
\end{conv}

Define
\[
\kkgr(L(E), L(F))_1 = \{\xi \in \kkgr(L(E), L(F)) : \ev(\xi) \text{
is an pointed preordered module morphism}\}
\]
and write $[L(E), L(F)]_1$ for the set of graded 
unital algebra homomorphisms $L(E) \to L(F)$ modulo graded polynomial homotopy. Our
next objective is to study the map
\begin{equation}\label{def:ucanmap}
\overline j \colon [L(E), L(F)]_1 \to \kkgr(L(E), L(F))_1.
\end{equation} 

To proceed further in the understanding of $\kkgr(L(E), L(F))_1$, we first need
to understand the $K_1^{\gr}$ group of a Leavitt path algebra. To do this, we 
first establish some remarks on ultramatricial
algebras and corner skew Laurent polynomial rings.

\subsection{\topdf{$K_1$}{K1} of ultramatricial algebras}\label{subsec:ultra}
As pointed out in \ref{rmk:lpa-str}, when $E$ is a regular graph its associated
Leavitt path algebra is strongly graded and thus $KH_1^{\gr}(L(E)) = KH_1^{\gr}(L(E)_0)$.
If in addition $\ell$ is a field, then we may replace $K$ for $KH$. This allows us to compute the graded $K$-theory of $L(E)$
in terms of its subalgebra of 
homogeneous elements of degree zero. 

The advantage of
this passage to $L(E)_0$ is that the latter algebra is 
\emph{ultramatricial}, that is, it is a countable increasing union of matricial algebras. 
For this reason, we wish
to prove some generalities regarding the first $K$-theory group of a unital
ultramatricial algebra. 
In what follows we will write $G_{\ab}$ for the abelianization of a group. In particular, if $R$ is a unital ring then $K_1(R) = GL(R)_{\ab}$. The field of two elements will be denoted $\F_2$.

The following observation is straightforward from the definition 
of ultramatricial algebra.

\begin{lem} \label{lem:ultra-comp}
Let $F, G \colon \Alg \to \cat{Grp}$ be two additive 
functors which preserve finite products and filtering colimits and let~$\eta \colon F \Rightarrow G$
be a natural transformation. 
The following statements are equivalent:
\begin{enumerate}[i)]
    \item For each unital ultramatricial algebra $R$, the map $\eta_R$ is an isomorphism.
    \item For each $n \in \N$, the map $\eta_{M_n(\ell)}$ is an isomorphism.
\end{enumerate}
\qed
\end{lem}

\begin{prop} \label{prop:k1-ultra}
Assume that $\ell$ is a field different from $\F_2$. If $R$ is a unital ultramatricial $\ell$-algebra, then the canonical map $R^\times_{\ab} \to K_1(R)$ is an isomorphism. 
\end{prop}
\begin{proof} By Lemma \ref{lem:ultra-comp}, it suffices to prove so for $R = M_n(\ell)$ for each $n \in \N$. Notice that $M_n(\ell)^\times = GL_n(\ell)$ and $[GL_n(\ell), GL_n(\ell)] = SL_n(\ell)$, since $\ell \neq \F_2$.  

We know that the (non-unital) inclusion 
of $\ell$ in the top-left corner induces an isomorphism in $K_1$, mapping 
$\lambda \in \ell^\times = K_1(\ell)$ to the class of $[I_n-\elmat_{1,1}+\lambda\elmat_{1,1}]$.
To conclude, we note that this isomorphism factors as the 
inverse of the 
determinant induced map $\det \colon GL_n(\ell)/SL_n(\ell) \to \ell^\times$
followed by the comparison map $GL_n(\ell)/SL_n(\ell) = M_n(\ell)^\times_{\ab} \to K_1(M_n(\ell))$.
\end{proof}

Given a functor $F \colon \Alg \to \cat{Grp}$, write 
\[
F^0(A) = F(\ev_1)(\ker(F(A[t]) \xto{F(\ev_0)} F(A))).
\]

Note that, since $\ev_1 \colon A[t] \to A$ is a retraction for any algebra $A$, the homomorphism $F(\ev_1)$ is a retraction; in particular it is surjective.
Hence, it maps the normal subgroup $\ker(F(\ev_0))$ of $F(A[t])$
to a normal subgroup of $F(A)$. This justifies the fact that
\[
\pi_0 F(A) := F(A)/F^0(A)
\]
is a group. Further, this assigment can be extended to a functor $\Alg \to \cat{Grp}$ by the universal properties of kernels and images.

In \cite{kklpa2}*{Proposition 2.8} Cortiñas and Montero show that the
Karoubi-Villamayor $K_1$-group of a purely 
infinite simple ring $R$ can be computed as $\pi_0 R^\times = R^\times/\{u(1) : u \in (R[t])^\times, u(0) = 1 \}$.
In our context, we obtain a similar conclusion for ultramatricial algebras. First, we need a lemma.

\begin{lem} If $\ell$ is a field, then
$(M_n(\ell)^\times)^{0} = SL_n(\ell)$.
\end{lem}
\begin{proof} We prove both inclusions. Since $SL_n(\ell)$
is generated by elementary matrices, it suffices to show that
$I_n + \lambda\elmat_{i,j} \in (M_n(\ell)^\times)^{0}$ for each $\lambda \in \ell^\times$, for which it suffices to consider the elementary matrix
\[
I_n + \lambda t\elmat_{i,j} \in GL_n(\ell[t]) = M_n(\ell[t])^\times = (M_n(\ell)[t])^\times.
\]
For the converse, let $u \in M_n(\ell)[t] = M_n(\ell[t])$ be an invertible matrix
such that $u(0) = I_n$. We have to prove that $u(1) \in SL_n(\ell)$.
Since $\ell[t]$ is an integral domain, a matrix in $M_n(\ell[t])$ is 
a unit if and only if $\det(u) \in (\ell[t])^\times = \ell^\times$.
In particular $\det(u)$ is constant and thus
\[
\det(u(1)) = \det(u)(1) = \det(u)(0) = \det(u(0)) = \det(I_n) = 1.
\]
\end{proof}

\begin{prop} \label{prop:ultra-pi0}
Assume that $\ell$ is a field. 
If $\ell \neq \F_2 $, then the comparison map 
$R^\times_{ab} \to \pi_0 R^\times \to K_1(R)$ is 
an isomorphism for every unital ultramatricial algebra $R$.
If $\ell = \F_2$, then 
$K_1(R) = \pi_0 R^\times = 1$.
\end{prop}
\begin{proof} Since 
$(M_n(\ell)^\times)^{0} = SL_n(\ell)$ for each $n \in \N$, we have
\[
\pi_0(M_n(\ell)^\times) = GL_n(\ell)/SL_n(\ell) = \begin{cases}
GL_n(\ell)_{ab} & |\ell| > 2\\
1 & |\ell| = 2
\end{cases}
\]
This together with Lemma \ref{lem:ultra-comp} prove the first
part of the lemma and also that $\pi_0 R^\times = 1$
for each unital ultramatricial algebra $R$ over $\F_2$. It remains 
to see that $K_1(R)$ is also trivial when $\ell = \F_2$, which
follows from matricial stability since
$K_1(M_n(\F_2)) \cong K_1(\F_2) = \F_2^\times = 1$.
\end{proof}

\begin{defn}\label{def:R-phi}
Let $E$ be a finite graph and $\phi \colon L(E) \to R$
a graded unital homomorphism. Write
\[
R_{\phi} := \bigoplus_{e \in E^1} \phi(ee^\ast) R_0 \phi(ee^\ast).
\]
\end{defn}

\begin{lem} \label{lem:ultra-cor}
Assume that $\ell$ is a field. If $R$ is a unital ultramatricial algebra 
and $e \in R$ an idempotent, then $eRe$ is a unital ultramatricial algebra. 
\end{lem}
\begin{proof} Let $R =\bigcup_{n \ge 1} R_n$
be an increasing union of unital, matricial subalgebras. 
Since the union is increasing, we may assume
without loss of generality that $1_R, e \in R_1$;
consequently $eRe = \bigcup_{n \ge 1} eR_ne$. 

It suffices to show that each algebra $eR_n e$ is matricial, that is, 
to show that the corner of any idempotent in a matricial algebra is again matricial. Let $A= M_{k_1}(\ell) \times \cdots \times M_{k_N}(\ell)$ be a matricial algebra and $(e_1, \ldots, e_N) \in \Idem(A)$. Since 
$eAe = \prod_{i=1}^N e_i A_i e_i$, we may prove that
any corner of a matrix algebra is again a matrix algebra.
In other words, we may assume that $N = 1$.
Put $k = k_1$ and $e = e_1$. Since $\ell$ is a field and an idempotent matrix represents a linear projector on $\ell^k$, there is an invertible matrix $u$ such that $u^{-1} e u = \sum_{s=1}^j \varepsilon_{s,s} =: p_j$ for some 
$j \in \{0, \ldots, k\}$;  
Conjugation by $u$ maps $e$ to $p_j$ and $e M_k(\ell) e$ to
$p_j M_k(\ell) p_j$, which is isomorphic to $M_j(\ell)$.
\end{proof}

\begin{prop} Let $E$ be a primitive graph and $R$ a strongly graded algebra
such that $R_0$ is ultramatricial. 
Each graded unital algebra homomorphism $\phi \colon L(E) \to R$ 
induces an isomorphism 
\begin{equation}\label{def:map-idem-k1}
(R_\phi)^\times_{\ab} = \prod_{e \in E^1} (\phi(ee^\ast) R_0 \phi(ee^\ast))^\times_{\ab}  \to K_1(R_0)^{E^1} \cong K_1^{\gr}(R)^{E^1}, \
(z_e)_{e \in E^1} \mapsto ([1-\phi(ee^\ast) + z_e)])_{e \in E^1}.
\end{equation}
\end{prop}
\begin{proof} 
It follows from Lemma \ref{lem:ultra-cor} and Proposition \ref{prop:k1-ultra}
that $K_1(\phi(ee^\ast) R_0 \phi(ee^\ast))$
can be computed as $(\phi(ee^\ast) R_0 \phi(ee^\ast))^\times_{\ab}$ 
for all $e \in E^1$.
By Proposition \ref{prop:idem-primi},
we know that each element $ee^\ast$ is a 
full idempotent of $L(E)_0$ and so, 
since $\phi$ is a graded unital homomorphism, it follows
that $\phi(ee^\ast)$ is a full idempotent of $R_0$. 
Consequently, each corner 
inclusion $\phi(ee^\ast) R_0 
\phi(ee^\ast) \to R_0$ induces an isomorphism at the level of 
$K_1$ groups. This concludes the proof.
\end{proof}

\subsection{The shift action for corner skew Laurent polynomials}

Let $R$ be a \emph{corner skew Laurent polynomial ring}, that is, a $\Z$-graded 
ring together with elements $t_+ \in R_1$, $t_- \in R_{-1}$ satisfying $t_- t_+ = 1$  (\cite{skew}*{Lemma 2.4}). Our motivating example is that of the Leavitt 
path algebra of an essential graph \cite{towards}*{p. 210}; indeed, if one selects one edge $e_v$ with range $v$ for each $v \in E^0$ then the elements
\[
t_+ = \sum_{v \in E^0} e_v, \qquad  t_l = \sum_{v \in E^0} e_v^\ast.
\]
yield a corner skew Laurent polynomial ring structure on $L(E)$.

Hazrat proves in \cite{hazbook}*{Proposition 1.6.6} that $p:=t_+t_-$ is a full idempotent if and only if $R$ is strongly graded. When this is the case, we know that $K_*^{\gr}(R)$
is naturally isomorphic to $K_*(R_0)$; we wish to understand to what kind of action the shift action translates 
to when viewed on $K_\ast(R_0)$.

In the case of the Leavitt path algebra of an essential graph $E$, Ara and Pardo prove in \cite{towards}*{Lemma 3.6} that the action on $L(E)_0$ is the one induced
by the (non-unital) corner homomorphism
\[
\alpha \colon L(E)_0 \to L(E)_0, \qquad x \mapsto t_+ x t_-.
\]
The proof relies on the fact that for $L(E)_0$, 
the Grothendieck group is generated by $1 \times 1$-idempotents. In other words, instead
of considering idempotents for all finite matrices, it suffices 
to do so for the ones of size $1$.

\begin{rmk}
We remark that in loc. cit. 
the shift agrees with the inverse of the map induced by $\alpha$.
This difference is explained by the fact that, if one builds $K_0^{\gr}$
for right modules instead of left modules, the isomorphism maps the shift 
action of $\sigma$ to that of $\sigma^{-1}$.
\end{rmk}

We will extend this result to $K_1$, for which we will use that
$K_1(L(E)_0)$ is generated by units. Since this is true
for any ultramatricial algebra, as noted in Proposition \ref{prop:k1-ultra}, 
we can in fact extend this to a statement on any strongly graded
corner skew Laurent polynomial
ring whose degree zero subring is ultramatricial. Namely, 
we prove the following:

\begin{thm} \label{thm:cslp-action}
Let $(R, t_+, t_-)$ be a strongly graded, corner skew Laurent polynomial ring. Consider $\alpha \colon R_0 \to R_0$ the homomorphism given by 
$x \mapsto t_+ x t_-$ and put $p =: \alpha(1)$. For any $x \in R^\times$, 
write $(R_0,x)$ for the class in $K_1$ given by the (left) module $R_0$
together with right multiplication by $x$. We have the following is an equality on $K_1(R_0)$:
\[
[(R_1, x)] = [(R_0, 1+p-\alpha(x)].
\]
\end{thm}
\begin{proof} Observe that $R_1 = R_0 t_+$ and that right multiplication
by $t_-$ yields an isomorphism $R_1 \iso R_0 p$ whose inverse
in right multiplication by $t_+$. It follows that 
$[(R_1, x)] = [(R_0p, \alpha(x))]$. Now, the following diagram 
with exact rows
\[
\begin{tikzcd}[column sep = huge, row sep = large]
    R_0 (p-1) \arrow[hook]{r} \arrow[equals]{d}
    & R_0 \arrow{r}{- \cdot p} \arrow{d}{- \cdot (1-p+\alpha(x))} & 
    R_0 p \arrow{d}[right]{\alpha(x)}\\
    R_0 (p-1) \arrow[hook]{r} & R_0 \arrow{r}{- \cdot p} & R_0 p\\
\end{tikzcd}
\]
says that
\[
[(R_0, 1+p-\alpha(x))] = [(R_0(p-1), 1)] + [(R_0p, \alpha(x))]
 = [(R_0 p, \alpha(x))].
\]
\end{proof}

We remark that the following corollary applies to Leavitt path
algebras of essential graphs, which is our main interest for this result.

\begin{coro} \label{coro:d-shift-d=alpha}
Let $(R, t_+, t_-)$ be a strongly graded, corner skew Laurent polynomial ring. Assume that $\ell$ is a field and $R_0$ is a unital ultramatricial algebra.
Writing $\alpha \colon R_0 \to R_0$ for the homomorphism given by 
$x \mapsto t_+ x t_-$ and $\mathsf{Dade} \colon K_1(R_0) \to K_1^{\gr}(R)$ for the isomorphism of \eqref{map:dade}, 
the following diagram is commutative:
\[
\begin{tikzcd}
    K_1^{\gr}(R) \arrow{r}{\sigma} & K_1^{\gr}(R)\\
    K_1(R_0) \arrow{u}{\mathsf{Dade}} \arrow{r}{K_1(\alpha)} & \arrow{u}[right]{\mathsf{Dade}} K_1(R_0)
\end{tikzcd}
\]
\end{coro}
\begin{proof} By Proposition \ref{prop:k1-ultra} applied to $R_0$,
an element in $K_1(R_0)$ can be represented 
by the class of the free module $R_0$
together with the automorphism $\rho_u \colon R_0 \to R_0$
of right multiplying by a unit $u$; we denote this by $(R_0, u)$.

The map $\mathsf{Dade}$ is induced by tensoring by $R \otimes_{R_0} -$.
Applying it to $(R_0, u)$ gives $(R \otimes_{R_0} R_0, u)$,
which is equal in $K_1^{\gr}(R)$ to the class of $(R, u)$, via the canonical
(graded) isomorphism $R \otimes_{R_0} R_0 \cong R$. The shift functor
maps $(R,u)$ to $(R[+1], u)$. Finally, the inverse
of the map $\mathsf{Dade}$ takes the class of the latter to its 0-th component, resulting
in $[(R_1, u)]$. This is exactly the action induced on units by $\alpha$,
as shown in Theorem \ref{thm:cslp-action}.
\end{proof}

\begin{rmk} \label{rmk:cslp-maps}
Let $(R,t_+,t_-)$ be a corner skew Laurent
polynomial ring, and let 
$f \colon R \to S$ be a unital graded algebra homomorphism. 
If we put $s_- = f(t_-)$, $s_+ = f(t_-)$, then $(S,s_+,s_-)$
is a corner skew Laurent polynomial ring. Further, if $R$
is strongly graded then so is $S$ (\cite{hazbook}*{Proposition 1.1.15 (4)}).
\end{rmk}

\subsection{Surjectivity of 
the map \topdf{\eqref{def:ucanmap}}{j}}

We are now in position to prove that the map \eqref{def:ucanmap} is surjective. 
To do this, we first define a certain modification
of a graded algebra map $L(E) \to R$ by an element of $K_1(R_0)$, cf. \cite{kklpa2}*{Equation 5.10} and \cite{kkh}*{Equation 13.6}.

\begin{defn} Let $E$ be a primitive 
graph. Given $\phi \colon L(E) \to R$ a graded algebra homomorphism, we define the following group epimorphism
\begin{align}
U \colon &(R_\phi)^\times_{\ab}
\xto{\eqref{def:map-idem-k1}} K_1(S_0)^{E^1} \xto{s_\ast} K_1(S_0)^{E^0}
\to \gBF^\vee(E) \otimes_{\Z[\sigma]} KH_1^{\gr}(S)  \label{def:U} \\
&z \longmapsto \prod_{e \in E^1} s(e) \otimes (1-\phi(ee^\ast)+z_e). \notag
\end{align}
Given $z = (z_e)_{e \in E^1} \in (R_\phi)^\times$, we associate to it
a graded unital map $\phi_z \colon L(E) \to R$ defined by
\begin{equation}\label{eq:phiz}
\phi_z \colon L(E) \to L(F), \quad \phi(e) = z_e\phi(e), \quad \phi_z(e^\ast) = \phi(e^\ast)z_e^{-1}, \quad \phi_u(v) = \phi(v) 
\ (v \in E^0, e\in E^1).
\end{equation}
\end{defn}

\begin{lem} \label{lem:jf=dx+xi}
If $\xi \in \kkgr(L(E), L(F))_1$, then 
there exist
a graded unital algebra map $\phi \colon L(E) \to L(F)$ and 
$x \in \gBF^\vee(E) \otimes_{\Z[\sigma]} KH_1^{\gr}(L(F))$
such that 
\[
\overline j(\phi) + d(x) = \xi.
\]
\end{lem}
\begin{proof} By \cite{lift}*{Theorem 6.1} (see also \cite{vas}*{Theorem 3.2}),
there exists a unital graded algebra map $\phi \colon L(E) \to L(F)$
such that $K_0^{\gr}(\phi) = \ev(\xi)$. Since $K_0^{\gr}(\phi) = \ev\overline{j}(\phi)$, this says that $\overline j(\phi)-\xi \in \ker(\ev) = \im d$, from which the lemma now follows. 
\end{proof}

\begin{lem} \label{lem:jf+u=jfu}
Let $E$ be a primitive graph and
let $R$ be a graded algebra such that $R_0$ is a unital ultramatricial
algebra. If $\phi \colon L(E) \to R$ is a unital graded homomorphism, then for each $z = (z_e)_{e \in E^1} \in R_\phi^\times$
we have 
\[
d(U([z])) + j(\phi) = j(\phi_z).
\]
\end{lem}
\begin{proof} 
Employing the notation of \eqref{def:uf}, one checks that $u_{\phi_z} = u_\phi \cdot U(z)$. Part iv) of Theorem \ref{thm:uct} and Corollary \eqref{coro:uf} imply that $j(\phi_z) = j(\phi) + d(U([z]))$.
\end{proof}

\begin{coro} \label{coro:surjectivity}
The map \eqref{def:ucanmap} is surjective.
\end{coro}
\begin{proof} Let $\xi \in \kkgr(L(E), L(F))_1$. By Lemma \ref{lem:jf=dx+xi}
there exists $\phi \colon L(E)\to L(F)$ a graded unital algebra homomorphism 
and $x \in \gBF^\vee(E) \otimes_{\Z[\sigma]} KH_1^{\gr}(L(F))$
such that $\overline j(\phi) + d(x) = \xi$. Since \eqref{def:U}
is surjective, there exists $z \in L(F)_\phi$ such that $x = [U(z)]$,
and thus $\xi = \overline{j}(\phi_z)$ by Lemma~\ref{lem:jf+u=jfu}.
\end{proof}

\subsection{Injectivity of the map
\topdf{\eqref{def:ucanmap}}{j}}

Next we analyze the injectivity of $\overline j$. 
We now consider the set $[L(E), L(F)]_{1, M_2}$
of graded unital maps $L(E) \to L(F)$ up to graded $M_2$-homotopy.
As we shall presently see, restricted to this quotient of $[L(E), L(F)]_1$
the map $\overline j$ becomes injective. 

\begin{lem}\label{lem:v~u=>phiv~phiu}
With notation as in Lemma \ref{lem:jf+u=jfu},
if $v,u \in R_\phi^\times$
are such that $[v] = [u]$ in $(R_\phi)^\times_{\ab}$,
then $\phi_v \sim \phi_u$.
\end{lem}
\begin{proof} By Proposition \ref{prop:ultra-pi0}, there 
exist 
\[
(Z_e)_{e \in E^1} \in \prod_{e \in E^1} (\phi(ee^\ast) R_0 \phi(ee^\ast))[t])^\times = \prod_{e \in E^1} (\phi(ee^\ast) (S[t])_0 \phi(ee^\ast)))^\times
\]
such that $Z_e(0) = u_e$ and $Z_e(1) = v_e$.
If we compose $\phi$ with the inclusion $i \colon S \to S[t]$
and then consider $h := (i \circ \phi)_Z$, it follows that 
$ev_i \circ h = \phi_{Z(i)}$; this concludes the proof.
\end{proof}

\begin{lem}\label{lem:action-s(e)=r(f)}
Let $E$ be a primitive graph and $\phi \colon L(E) \to R$ a graded unital homomorphism. Assume that 
$R_0$ is ultramatricial. Given $e,f \in E^1$
such that $r(f) = s(e)$
and $u \in \phi(ee^\ast)R_0\phi(ee^\ast)^\times$, we have that
\[
\sigma \cdot [1-\phi(ee^\ast)+u] = 
[1-\phi(fe(fe)^\ast) + \phi(f)u\phi(f^\ast)]
\]
in $K_1(R_0)$. In particular, for any other $g \in E^1$ such that
$r(g) = s(e)$ we have
\[
[1-\phi(fe(fe)^\ast) + \phi(f)u\phi(f^\ast)]
=[1-\phi(ge(ge)^\ast) + \phi(g)u\phi(g^\ast)].
\]
\end{lem}
\begin{proof} For each vertex $v \in E^0 \setminus \{s(e)\}$,
let $f_v$ be an edge with range $v$; their existence
is guaranteed by the essentiallity hypothesis on $E$.
Set $f_{s(e)} := f$. As pointed out in \cite{towards}*{p. 203}, 
the elements $t_+ = \sum_{v \in E^0} f_v$
and $t_- = t_+^\ast$ satisfy $t_- t_+ = 1$,
yielding a corner skew Laurent polynomial structure on $L(E)$. 
Further, this gives such a structure on $R$
via the elements $\phi(t_+)$ and $\phi(t_-)$.

As per Theorem \ref{thm:cslp-action}, 
the action of $\sigma$ on $K_1(R_0)$ can be described as
the one induced by the automorphism
\[
\alpha \colon R_0 \to R_0, \qquad x \mapsto \phi(t_+)x\phi(t_-).
\]
Hence
\[
\sigma \cdot [1-\phi(ee^\ast)+u] = [1-\phi(t_+t_-)+\phi(t_+)(1-\phi(ee^\ast)+u)\phi(t_-)]
= [1- \phi(t_+)(\phi(ee^\ast)-u)\phi(t_-)].
\]
Since $\phi(ee^\ast), u \in \phi(ee^\ast)R_0\phi(ee^\ast)$, it follows that
$f_v e = 0$ unless $v = s(e)$, that is, unless $f_v = f$.
Thus
\[
[1- \phi(t_+)(\phi(ee^\ast)-u)\phi(t_-)] =
[1-\phi(fe(fe)^\ast) + \phi(f)u\phi(f^\ast)],
\]
concluding the proof.
\end{proof}

\begin{lem} \label{lem:ad-lemma}
Let $E$ be a primitive
graph and $R$ a graded
algebra such that $R_0$ is an ultramatricial algebra. 
Let $\phi \colon L(E) \to R$ be a graded algebra map.
If $z \in \prod_{e \in E^1} (\phi(ee^\ast)R_0 \phi(ee^\ast))^\times$
is such that $d(U([z])) = 0$, then $\phi \sim_{\ad} \phi_z$.
\end{lem}
\begin{proof} Consider the  homomorphism
\[
\lambda \colon R_\phi \to R_\phi, \qquad a 
\mapsto \sum_{e \in E^1} \phi(e)a\phi(e^\ast) 
\]
As in the ungraded case, writing $B_{e,f} = 
\delta_{r(e), s(f)}$ and using Lemma \ref{lem:action-s(e)=r(f)}, one checks
that the following square is commutative
\[
\begin{tikzcd}
K_1(R_\phi) \arrow{d}{\sim} \arrow{r}{\lambda} & K_1(R_\phi) \arrow{d}{\sim}\\
K_1(R)^{E^1} \arrow{r}{\sigma B} & K_1(R)^{E^1} 
\end{tikzcd}
\]
Writing $E_s$ for the graph with adjacency matrix $B$, 
we get a commutative diagram
\[
\begin{tikzcd}
K_1(R_\phi) \arrow{d}{\sim} \arrow{r}{1-\lambda} & K_1(R_\phi) \arrow{d}{\sim} & \\
K_1(R)^{E^1} \arrow{d}{s_\ast}\arrow{r}{I-\sigma B} & \arrow{d}{s_\ast}K_1(R)^{E^1} \arrow{r} & \gBF^\vee(E_s) \otimes_\Z KH_1^{\gr}(R)\arrow{d}& \\
K_1(R)^{E^0} \arrow{r}{I-\sigma A_E} & K_1(R)^{E^0} \arrow{r} & \gBF^\vee(E) \otimes_\Z KH_1^{\gr}(R) \arrow{r} & \kkgr(L(E), R)
\end{tikzcd}
\]
The rightmost horizontal map is injective by Theorem \ref{thm:uct}. Further,
the map induced by $s_\ast$ at the level of dual Bowen-Franks modules is 
an isomorphism; its inverse is induced by $r^\ast$. 

Now, since $\partial(U([z])) = 0$, there exists $[\nu] \in K_1(S_\phi)$
such that $[\nu \lambda(\nu)^{-1}] = [z]$.
By Lemma \ref{lem:v~u=>phiv~phiu}, this
says that $\phi_z \sim \phi_{\nu \lambda(\nu)^{-1}}$. 
Finally, like in the ungraded case, one checks that $\phi_{\nu \lambda(\nu)^{-1}} =
\ad(v) \circ \phi$.
\end{proof}

\begin{thm} \label{thm:jeq-ad-m2}
Assume that $\ell$ is a field. Let $E$ and $F$ be two primitive graphs. Given two unital graded homomorphisms $f, g \colon L(E) \to L(F)$, the following statements are equivalent:
\begin{itemize}
\item[(i)] $j(f) = j(g)$;
\item[(ii)] $f \sim_{\ad} g$;
\item[(iii)] $f \sim_{M_2} g$.
\end{itemize}
\end{thm}
\begin{proof} The fact that (ii) implies (iii) is the content of Lemma \ref{lem:ad->m2}. Since $j$ is matricially stable and homotopy invariant, 
it follows that (iii) implies (i); hence, it remains to show that (i)
implies (ii). Suppose that $j(f) = j(g)$. In particular $K_0^{\gr}(f) = \ev \circ j(f)$ agrees with $K_0^{\gr}(g)$ and therefore, by \cite{lift}*{Corollary 3.5}, 
there exists a homogeneous unit of degree zero $u \in R$
such that $\ad(u) \circ f$ and $g$ agree on $D(E)_1 = \mathrm{span}_\ell\{ee^\ast : e \in E^1\} = \mathrm{span}_\ell\{ee^\ast, v : e \in E^1, v \in E^0\}$. Since $\ad(u) \circ f \sim_{\ad} f$, we may
without loss of generality assume that $f$ and $g$ agree on $D_1(E)$.

Now, if we put $z_e = g(e)f(e^\ast)$,
for each $e \in E^1$, these are units in each corner $f(ee^\ast) L(F)_0 f(ee^\ast)$ with inverse $f(e)g(e^\ast)$. Using the notation of Lemma \ref{lem:jf+u=jfu}, it follows that $g = f_z$ and $d(U([z])) = j(f_z)-j(f) = j(g)-j(f) = 0$. Thus, we can apply Lemma \ref{lem:ad-lemma} to obtain that $f \sim_{\ad} g$.
\end{proof}

\begin{coro} \label{coro:M2-htpy=hom}
Assume that $\ell$ is a
field. If $E$ and $F$ are two primitive graphs, 
then the map $\overline j \colon [L(E), L(F)]_{1,M_2} \to \kkgr(L(E), L(F))_1$ is bijective.
\qed
\end{coro}
\begin{proof} Surjectivity was proven in Corollary \ref{coro:surjectivity}; injectivity 
follows from Theorem \ref{thm:jeq-ad-m2}.
\end{proof}

\section{Homotopy classification} \label{sec:clasi}

We conclude with a homotopy classification 
theorem. To simplify its statement, we shall 
say that two algebras are \emph{unitally graded homotopy equivalent}
if there exists a unital homotopy equivalence between them 
whose homotopy inverse is also a unital homorphism.

\begin{thm} \label{thm:htpy-clasif}
Let $\ell$ be a field 
and $E$ and $F$ two primitive graphs. 
The following statements are equivalent:
\begin{itemize}
    \item[(i)] The pointed, preordered $\Z[\sigma]$-modules $(\gBF(E),\gBF(E)_+,1_E)$ 
    and $(\gBF(F), \gBF(E)_+, 1_F)$ are isomorphic.
    \item[(ii)] There exists an isomorphism $\xi \in \kkgr(L(E), L(F))_1$.
    \item[(iii)] The algebras $L(E)$ and $L(F)$ are graded unitally homotopy equivalent.
\end{itemize}
\end{thm}
\begin{proof} In \cite{arcor}*{Theorem 13.1} it was proved that, in the surjection of Theorem \ref{thm:uct}, one can lift isomorphisms at the level of $K_0^{\gr}$ to $\kkgr$-isomorphisms. Since by definition a lifting of a pointed
preordered module map lies in $\kkgr(L(E), L(F))_1$; this proves implication $(i) \Rightarrow (ii)$. 

Next assume $(ii)$ and consider the inverse
of $\xi$, noting that $\xi^{-1} \in \kkgr(L(F), L(E))_1$. By 
Corollary \ref{coro:M2-htpy=hom}, there exist unital algebra homomorphisms 
$f \colon L(E) \to L(F)$ and $g \colon L(F) \to L(E)$
such that $j(f) = \xi$ and $j(g) = \xi^{-1}$. In particular 
$j(fg) = \id_{L(F)}$ and $j(gf) = \id_{L(E)}$, and thus 
Theorem \ref{thm:jeq-ad-m2} says that there exist 
units $u \in L(F)_0$, $v \in L(E)_0$ such that $fg \sim \ad(u)$ and 
$gf \sim \ad(v)$. This readily implies that $f$ is a homotopy equivalence. We have thus proved that $(ii)$ implies~ $(iii)$.

Finally, we prove that $(iii)$ implies $(i)$. Let $f \colon L(E) \to L(F)$ be a unital homotopy equivalence. The map $K_0^{\gr}(f)$ is a pointed, preordered module map between the Bowen-Franks modules of $E$ and $F$, it suffices to see that it is an isomorphism. To prove this, we note that
$K_0^{\gr}$ agrees with $KH_0^{\gr}$ for $L(E)$ and $L(F)$ as per Remark \ref{rmk:htpy-khgr} and that $KH_0^{\gr}$ maps homotopy equivalences to isomorphisms. This concludes the proof.
\end{proof}

By Theorem \ref{thm:htpy-clasif}, the primitive case of Conjecture \ref{conj-hazrat} is equivalent to the following:

\begin{conj} Let $\ell$ be a field. If $E$ and $F$ are primitive graphs, then $L_\ell(E)$
and $L_\ell(F)$ are graded isomorphic if and only if they are unitally graded homotopy equivalent.
\end{conj}

\end{document}